\documentclass[11pt,a4paper,english]{article}

\newcommand{\beq}{\begin{equation}}
\newcommand{\eeq}{\end{equation}}

\usepackage[latin1]{inputenc}
\usepackage{babel}
\usepackage[centertags]{amsmath}
\usepackage{amsfonts}
\usepackage{amssymb}
\usepackage{color}
\usepackage{amsthm}
\usepackage{newlfont}
\usepackage{graphicx}
\usepackage{subfigure}
\usepackage{dsfont}
\usepackage{bbm}

\DeclareMathOperator*{\argsup}{arg\,sup}

\newtheorem{theorem}{Theorem}[section]
\newtheorem{lemma}[theorem]{Lemma}

\newtheorem{proposition}[theorem]{Proposition}
\newtheorem{definition}[theorem]{Definition}
\newtheorem{remark}[theorem]{Remark}
\newtheorem{Assumptions}[theorem]{Assumption}

\newcommand{\PP}{ \mathbb{P}}

\newcommand{\LL}{ \mathcal{L}}
\makeatletter  
\@addtoreset{equation}{section}

\makeatother  
\makeatletter
\renewcommand*\@biblabel[1]{}
\makeatother

\begin{document} 
\title{\textbf{On a Strategic Model of Pollution Control}}
\author{Giorgio Ferrari\thanks{Center for Mathematical Economics, Bielefeld University, Germany; \texttt{giorgio.ferrari@uni-bielefeld.de}} \and Torben Koch\thanks{Corresponding author. Center for Mathematical Economics, Bielefeld University, Germany; \texttt{t.koch@uni-bielefeld.de}}}
\date{\today}
\maketitle

\vspace{0.5cm}

{\textbf{Abstract.}} This paper proposes a strategic model of pollution control. A firm, representative of the productive sector of a country, aims at maximizing its profits by expanding its production. Assuming that the output of production is proportional to the level of pollutants' emissions, the firm increases the level of pollution. The government of the country aims at minimizing the social costs due to the pollution, and introduces regulatory constraints on the emissions' level, which then effectively cap the output of production. Supposing that the firm and the government face both proportional and fixed costs in order to adopt their policies, we model the previous problem as a stochastic impulse two-person nonzero-sum game. The state variable of the game is the level of the output of production which evolves as a general linearly controlled one-dimensional It\^o-diffusion. Following an educated guess, we first construct a pair of candidate equilibrium policies and of corresponding equilibrium values, and we then provide a set of sufficient conditions under which they indeed realize an equilibrium. Our results are complemented by a numerical study when the (uncontrolled) output of production evolves as a geometric Brownian motion, and the firm's operating profit and the government's running cost functions are of power type. An analysis of the dependency of the equilibrium policies and values on the model parameters yields interesting new behaviors that we explain as a consequence of the strategic interaction between the firm and the government.

\smallskip


\smallskip

{\textbf{OR/MS subject classification}}: Environment: pollution; Games/group decisions: stochastic; Probability: stochastic model applications; Dynamic programming/optimal control: Markov.

\smallskip

{\textbf{JEL subsject classification}}: C61; C73; Q52.

\smallskip

{\textbf{MSC2010 subsject classification}}: 93E20; 91B70; 91A15; 91B76.

\section{Introduction}
\label{introduction}

In recent years, the growing importance of global environmental issues, such as the global warming, pushed countries or institutions to adopt environmental policies aiming at reducing the level of pollution. Some of these policies are the result of international agreements (such as the Kyoto Protocol, or the Paris Climate Agreement of 2016); some others are adopted more on a local scale: it is indeed a news of December 2016 that the authorities of Beijing issued a five-day warning and ordered heavy industries to slow or halt their production due to increasing smog.\footnote{See, e.g., https://www.theguardian.com/world/2016/dec/17/beijing-smog-pollution-red-alert-declared-in-china-capital-and-21-other-cities.}

Environmental problems have attracted the interest of the scientific community as well (see, e.g, Nordhaus (1994), and Chapter 9 of Perman et al.\ (2003) for an exhaustive introduction to pollution control policies). Many papers in the mathematical and economic literature take the point of view of a social planner to model the problem of reducing emissions of pollutants arising from the production process of the industrial sector. For example, in Pindyck (2000) and Pindyck (2002) a social planner aims at finding a time at which the reduction of the rate of emissions gives rise to the minimal social costs. In Pommeret and Prieur (2013) the optimal environmental policy to be adopted is the one that maximizes the economy's instantaneous net payoff, i.e.\ the sum of the economic damage of pollution and of the economic benefits from production. Finally, Goulder and Mathai (2000) and Schwoon and Tol (2006) consider the planner's problem of choosing the abatement policy, and research and development investment, that minimize the costs of achieving a given target of CO$_2$ concentration. All those works tackle the resulting mathematical problems with techniques from (stochastic) optimal control theory, and provide policy recommendations.

In this paper we do not take the point of view of a fictitious social planner, but we propose a \emph{strategic model} of pollution control. An infinitely-lived profit maximizing firm, representative of the productive sector of a country, produces a single good, and faces fixed and proportional costs of capacity expansion. In line with other papers in the environmental economics literature (cf.\ Pindyck (2002) and Pommeret and Prieur (2013)), we suppose that the output of production is proportional to the level of pollutants' emissions. Those are negatively perceived by the society, and we assume that the social costs of pollution can be measured by a suitable penalty function. A government (or a government environmental agency) intervenes in order to dam the level of emissions, e.g., by introducing regulatory constraints on the emissions' level, which then effectively cap the output of production. We suppose that the interventions of the government have also some negative impact on the social welfare (e.g., they might cause an increase in the level of unemployment or foregone taxes), and  we assume that such negative externality can be quantified in terms of instantaneous costs with fixed and proportional components. The government thus aims at minimizing the total costs of pollution and of the interventions on it. 

Due to the fixed costs of interventions faced by the firm and the government, it is reasonable to expect that the two agents intervene only at discrete times on the output of production. Between two consecutive intervention times, the latter is assumed to evolve as a general regular one-dimensional It\^o-diffusion\footnote{Uncertain capital depreciation or technological uncertainty might justify the stochastic nature of the output of production (see also Asea and Turnovsky (1998), Eaton (1981), Epaulard and Pommeret (2003) and W\"alde (2011)).}. We therefore model the previously discussed pollution control problem as a \emph{stochastic impulse\footnote{Stochastic impulse control problems naturally arise in many areas of applications. Among these we refer to optimal control of exchange and interest rates (Cadenillas and Zapatero (1999), Mitchell et al.\ (2014), Perera et al.\ (2016), among others), portfolio optimization with fixed transaction costs (Korn (1999)), optimal inventory control (Bensoussan et al.\ (2010), and Harrison et al.\ (1983)), rational harvesting of renewable resources (Alvarez (2004)), and optimal dividend problems (Cadenillas et al.\ (2006)).} nonzero-sum game} between the government and the firm. The policy of each player is a pair consisting of a sequence of times, and a sequence of sizes of interventions on the output of production, and each player aims at picking a policy that optimizes her own performance criterion, given the policy adopted by the other player. The two players thus interact strategically in order to determine an equilibrium level of the output of production, i.e.\ of the level of pollutants' emissions.

Following an educated guess, we first construct a couple of candidate equilibrium policies, and of associated equilibrium values. In particular, we suppose that the equilibrium policies adopted by the firm and the government are characterized by four constant trigger values: on the one hand, whenever the output of production falls below a constant threshold, we conjecture that it is optimal for the firm to push the output of production to an upper constant level; on the other hand, whenever the level of emissions reaches an upper threshold, the government should provide regulatory constraints which let the output of production jump to a constant lower value. It turns out that, by employing these policies, the two agents keep the output of production (equivalently, the level of pollutants' emissions) within an interval whose size is the result of their strategic interaction.

In order to choose those four trigger values we require that the agents' performance criteria associated to the previous policies are suitably smooth, as functions of the current output of production level. Namely, each agent imposes that her own candidate equilibrium value is continuously differentiable at her own trigger values. We then move on proving a verification theorem which provides sufficient conditions under which the previous candidate strategies indeed form an equilibrium. In particular, we show that if the solution of a suitable system of four highly nonlinear algebraic equations exists and satisfies a set of appropriate inequalities, then such a solution will trigger an equilibrium. Our results are finally complemented by a numerical study in the case of (uncontrolled) output of production given by a geometric Brownian motion. Also, we discuss the dependency of the trigger values and of the equilibrium impulses' size on the model parameters. This comparative statics analysis shows interesting new behaviors that we explain as a consequence of the strategic interaction between the firm and the government. As an example, we find, surprisingly, that the higher the fixed costs for the firm, the smaller the sizes of the impulses applied by both the agents on the production process.

The contribution of this paper is twofold. On the one hand, we propose a general strategic model that highlights the interplay between the productive sector and the government of a country for the management of the pollution which inevitably arises from the production process.\footnote{For other works modeling the pollution control problem as a dynamic game one can refer, among others, to the example in Section 4 of De Angelis and Ferrari (2016), Long (1992) and van der Ploeg and de Zeeuw (1991).}. On the other hand, from a mathematical point of view, ours is one of the first papers dealing with a two-player nonzero-sum stochastic impulse game. It is worth noticing that a verification theorem for two-player nonzero-sum stochastic impulse games, in which the uncontrolled process is a multi-dimensional It\^o-diffusion, has been recently proved in A{\"i}d et al.\ (2018). There the authors give a set of sufficient conditions under which the solutions (in an appropriate sense) of a system of coupled constrained PDE problems (the so-called quasi-variational inequalities, QVIs\footnote{The interested reader may refer to the book by Bensoussan and Lions (1984) for the theory of QVIs.}) identify equilibrium values of the game. Then, they consider a one-dimensional symmetric game with linear running costs, and obtain equilibrium values and equilibrium policies by finding the solutions of the related system of QVIs, and by verifying their optimality.

Our methodology is different with respect to that of A{\"i}d et al.\ (2018). Here we obtain candidate equilibrium values without relying on solving the system of QVIs that would be associated to our game. Indeed, our candidate equilibrium values are constructed as the performance criteria that the players obtain by applying a potentially suboptimal policy. This construction, which employs probabilistic properties of one-dimensional It\^o-diffusions, has been already used in single-agent impulse control problems (see, e.g., Alvarez (2004), Alvarez and Lempa (2008) and Egami (2008)), and has the advantage of providing candidate equilibrium values which are automatically continuous functions of the underlying state variable. As a computationally useful byproduct, in our asymmetric setting we only have to find the four equilibrium trigger values, and for that we only need four equations. This is in contrast to the eight equations one would obtain by imposing $C^0$ and $C^1$-regularity of the solutions to the system of QVIs (cf.\ A{\"i}d et al.\ (2018)).

The rest of the paper is organized as follows. In Section \ref{sec:setting} we introduce the setting and formulate the problem. In Section \ref{sec:representation} we construct candidate equilibrium policies and candidate equilibrium values, whereas in Section \ref{sec:verification} we provide a verification theorem. Finally, in Section \ref{sec:numeric} we provide the numerical solution to an example, and we study the dependency of the equilibrium with respect to the model parameters. Conclusions are finally drawn in Section \ref{conclusions}.


\section{Setting and Problem Formulation}
\label{sec:setting}

We consider a firm (agent $1$), and a government (agent $2$). The firm produces a single good, and its profits from production are described by a function $\pi:\mathbb{R}_+\mapsto\mathbb{R}_+$ which is continuous, strictly concave and increasing. We assume that the production process leads to emissions, for example of greenhouse gases such as CO$_2$, that are proportional to the level of the output (see also Pindyck (2002) and Pommeret and Prieur (2013), among others). These emissions have a negative externality on the social welfare, and the resulting disutility incurred by the society is measured by a cost function $C:\mathbb{R}_+\mapsto\mathbb{R}_+$ that depends on the rate of emissions. The function $C$ is continuous, strictly convex and increasing.

The production process is assumed to be stochastic, since it may depend on uncertain capital depreciation or other exogenous random factors (see also Asea and Turnovsky (1998), Bertola (1998), Epaulard and Pommeret (2003) and W\"alde (2011), among others). In particular, let $W=(W_t)_{t\geq0}$ be a one-dimensional, standard Brownian motion on a complete filtered probability space $(\Omega,\mathcal{F},\mathbb{F},\mathbb{P})$, where $\mathbb{F}:=(\mathcal{F}_t)_{t\geq 0}$ is a filtration satisfying the usual conditions. The output of production at time $t\geq 0$ is denoted by $X_t$, and it evolves as a linear It\^o-diffusion on $(0,\infty)$; that is,
\begin{align}
\label{X1}
dX_t= \mu(X_t) dt+\sigma(X_t)dW_t, \qquad X_0=x>0,
\end{align}
for some Borel-measurable functions $\mu,\sigma$ to be specified. Here, $\mu$ is the trend of the production, while $\sigma$ is a measure of the fluctuations around this trend.

To account for the dependency of $X$ on its initial level, from now on we shall write $X^x$ where appropriate, and $\mathbb{P}_x$ to refer to the probability measure on $(\Omega,\mathcal{F})$ such that $\PP_x(\,\cdot\,) = \mathbb{P}(\,\cdot\,| X_0=x)$, $x \in (0,\infty)$. Throughout this paper we will equivalently use the notations $\mathbb{E}[f(X^x_t)]$ and $\mathbb{E}_x[f(X_t)]$, $f:\mathbb{R} \to \mathbb{R}$ Borel-measurable and integrable, to refer to expectations under the measure $\mathbb{P}_x$.

For the coefficients of the SDE \eqref{X1} we make the following assumption, which will hold throughout the paper.
\begin{Assumptions}
\label{ass:D2}
The functions $\mu:\mathbb{R} \mapsto \mathbb{R}$ and $\sigma: \mathbb{R} \mapsto (0,\infty)$ are such that
\beq
\label{YamadaWatanabe}
|\mu(x)-\mu(y)| \leq K|x-y|,\qquad |\sigma(x)-\sigma(y)| \leq h(|x-y|), \qquad x,y \in (0,\infty),
\eeq
for some $K>0$, and $h:\mathbb{R}_+ \mapsto \mathbb{R}_+$ strictly increasing such that $h(0) = 0$ and
\beq
\label{YamadaWatanabe2}
\int_{(0,\varepsilon)}\frac{du}{h^2(u)} = \infty\,\,\,\,\mbox{for every}\,\,\varepsilon>0.
\eeq
\end{Assumptions}

As a consequence of the above assumption one has that if a solution to \eqref{X1} exists, then it is pathwise unique by the Yamada-Watanabe's Theorem (cf.\ Karatzas and Shreve (1998), Proposition $5.2.13$ and Remark $5.3.3$, among others). Moreover, from \eqref{YamadaWatanabe} and \eqref{YamadaWatanabe2} it follows that for every $x \in (0,\infty)$ there exists $\varepsilon>0$ such that
\beq
\label{LI}
\int_{x-\varepsilon}^{x+\varepsilon}\frac{1 + |\mu(y)|}{\sigma^{2}(y)}\,dy < +\infty.
\eeq
Local integrability condition \eqref{LI} implies that \eqref{X1} has a weak solution (up to a possible explosion time) that is unique in the sense of probability law (cf.\ Karatzas and Shreve (1998), Section $5.5$C). Therefore, \eqref{X1} has a unique strong solution (possibly up to an explosion time) due to Karatzas and Shreve (1998), Corollary $5.3.23$. One-dimensional diffusions like the geometric Brownian motion and the CIR process (under a suitable restriction on the parameters, i.e.\ the so-called Novikov's condition) satisfy the assumptions of our setting.

\begin{remark}
\label{remark1}
An example of microfoundation for a stochastic dynamics of the output of production is the following (cf.\ Bertola (1998)). Assume that at time $t\geq 0$ the output of production $X_t$ is given in terms of the capital stock, $K_t$, and the output of labor, $L_t$, by 
	\begin{align}\label{r1}
	X_t=\big(K_t^\rho L_t^{1-\rho}\big)^\gamma,\quad 0<\rho\leq 1,\text{ and }\gamma>0.
	\end{align}
	Also, suppose that the firm is faced with a constant elasticity demand function
	\begin{align}\label{r2}
	P_t=X_t^{\lambda-1},\quad 0<\gamma\lambda<1,
	\end{align}
	where $P_t$ is the product price at time $t\geq 0$, and $\lambda$ is a measure of the firm's monopoly power. Since the input of labor $L_t$ is chosen such that $L_t=\arg\max_L\big\{P_tX_t-wL\big\}$, for some wage $w>0$, one can obtain from \eqref{r1} and \eqref{r2} that
	\begin{align}
	\label{Lt}
	L_t=\bigg[\frac{\gamma\lambda}{w}(1-\rho)\bigg]^\frac{1}{1-(1-\rho)\gamma\lambda}K_t^\frac{\rho\gamma\lambda}{1-(1-\rho)\gamma\lambda}=\hat{\alpha} K_t^\frac{\rho\gamma\lambda}{1-(1-\rho)\gamma\lambda},
	\end{align}
	where $\hat{\alpha}:=\big[\frac{\gamma\lambda}{w}(1-\rho)\big]^\frac{1}{1-(1-\rho)\gamma\lambda}$. Hence, by plugging \eqref{Lt} into \eqref{r1} we have
	\begin{align}\label{r3}
	X_t=\hat{\alpha}^{(1-\rho)\gamma}K_t^{\frac{\gamma\rho}{1-(1-\rho)\gamma\lambda}}.
	\end{align} 
	If now capital stock is stochastic and depreciates at a rate $\delta>0$, i.e. $dK_t=-\delta K_tdt+\sigma K_tdW_t$ for some Brownian motion $W$ (see, e.g., W\"alde (2011)), by It\^o's formula one finds that $X_t$ evolves as  
	\begin{align*}
	dX_t=\hat{\mu}X_tdt+\hat{\sigma}X_tdW_t,
	\end{align*}
	for suitable constants $\hat{\mu},\hat{\sigma}$. 
\end{remark}

Both the agents can influence the process of production: on the one hand, the firm can instantaneously increase the level of production, for example by increasing the capital stock. This leads to instantaneous costs for the firm which have both a variable and a fixed component, and that we model through a function $g_1: \mathbb{R}_+\mapsto\mathbb{R}_+$ of the size of interventions on the production. In particular we take
  \begin{align*}
  g_1(\xi):=K_1+\kappa_1 \xi,\quad \xi \geq 0,
  \end{align*}
  with $K_1, \kappa_1>0$.
  On the other hand, the government can introduce regulatory constraints that effectively force the firm to decrease the level of production\footnote{Restrictions on the output of production can be achieved by the government in different ways. The interested reader may refer to the classical book by Pigou (1938).}, hence of the emissions. A similar situation has happened in December 2016 in Beijing where authorities issued a five-day warning and ordered heavy industries to slow or halt production in order to reduce the smog in the air. We assume that the instantaneous costs of a similar policy incurred by the government can be measured by a function $g_2:\mathbb{R}_+\mapsto\mathbb{R}_+$ given by
   \begin{align*}
   g_2(\eta):=K_2+\kappa_2 \eta, \quad \eta \geq 0,
   \end{align*}
   with $K_2, \kappa_2>0$. Such costs might arise because of an increase in the rate of unemployment or forgone taxes due to a possible decrease of the production capacity.\par
  Because of the presence of fixed costs, it is reasonable to expect that the firm (resp.\ the government) intervenes only at discrete times on the output of production by shifting the current level of output up (resp.\ down) of some nonzero amount. More formally, the policy of any agent is defined as follows. 
 \begin{definition}\label{defpol}
 The policies $\varphi_1$ and $\varphi_2$ of the firm and of the government, respectively, are pairs
 \begin{align*}
 \varphi_1&:=(\tau_{1,1},\dots,\tau_{1,n},\dots;\xi_{1},\dots,\xi_{n},\dots),\\
 \varphi_2&:=(\tau_{2,1},\dots,\tau_{2,n},\dots;\eta_{1},\dots,\eta_{n},\dots)
 \end{align*}
 where $0 \leq \tau_{i,1} \leq \tau_{i,2} \leq\dots$, for $i=1,2$, is an increasing sequence of $\mathbb{F}$-stopping times, $\xi_{n}$ are positive $\mathcal{F}_{\tau_{1,n}}$-measurable random variables, and $\eta_{n}$ are positive $\mathcal{F}_{\tau_{2,n}}$-measurable random variables.
 \end{definition}
 Intervening on the output of production, the two agents modify the dynamics of the production process which then becomes
  \begin{align}
	\label{controlledX}
  \,
  \begin{cases}
  \displaystyle X^{x,\varphi_1,\varphi_2}_t=&x+\int\limits_{0}^{t}\mu(X^{x,\varphi_1,\varphi_2}_s) ds+\int\limits_{0}^{t}\sigma (X^{x,\varphi_1,\varphi_2}_s) dW_s\\
  &\displaystyle +\alpha\sum\limits_{k:\tau_{1,k}\leq t}\xi_k\prod\limits_{l\geq 1}\mathbbm{1}_{\{\tau_{1,k}\neq\tau_{2,l}\}}-\sum\limits_{k:\tau_{2,k}\leq t}\eta_k,\quad t\geq0,\\
  \displaystyle X^{x,\varphi_1,\varphi_2}_{0-}=&x>0,
  \end{cases}  
  \end{align}
  where $\alpha>0$ measures the effect of an increase in the capital stock on the output of production, and $X^{x,\varphi_1,\varphi_2}_{t-}:=\lim_{\varepsilon \downarrow 0}X^{x,\varphi_1,\varphi_2}_{t -\varepsilon}$ for any $t\geq0$. 
  
  In \eqref{controlledX} $\xi_k$ represents the lump-sum increase of the output of production made by the firm at time $\tau_{1,k}$. Moreover, $\eta_k$ is the impact on production of the regulatory constraints imposed by the government at time $\tau_{2,k}$. If both the agents are willing to intervene on the output of production at the same time, it is reasonable to allow the government to have the priority: the infinite product $\prod\limits_{l\geq 1}\mathbbm{1}_{\{\tau_{1,k}\neq\tau_{2,l}\}}$ in \eqref{controlledX} takes care of that. In the rest of this paper we write $X^{x,\varphi_1,\varphi_2}$ to stress the dependence of the output of production on its initial level, and on the policies $\varphi_1$ and $\varphi_2$ adopted by the two agents. 
  \begin{remark}
  	Following the microfoundation of Remark \ref{remark1}, suppose that at a certain time $\tau_{1,k}$ the firm increases the capital stock by an amount $\xi_k$, while the government does not intervene. Then we have by \eqref{r3} that
  	\begin{align*}
  	X_{\tau_k}=\hat{\alpha}^{(1-\rho)\gamma}K_{\tau_k}^\frac{\rho\gamma}{1-(1-\rho)\gamma\lambda}=\hat{\alpha}^{(1-\rho)\gamma}\Big(K_{\tau_k-}+\xi_k\Big)^\frac{\rho\gamma}{1-(1-\rho)\gamma\lambda}.
  	\end{align*}
  	Taking $\gamma>1$, for $\rho=\frac{1-\gamma\lambda}{\gamma-\gamma\lambda}\in(0,1)$ and $\lambda$ such that $\gamma\lambda\in(0,1)$, we find
  	\begin{align*}
  	X_{\tau_k}= X_{\tau_k-}+\hat{\alpha}^{(1-\rho)\gamma}\xi_k,
  	\end{align*}
  	that is consistent with \eqref{controlledX} if we set $\alpha:=\hat{\alpha}^{(1-\rho)\gamma}$.
  \end{remark}

The firm's total expected profits arising from production, net of present costs, are
 \begin{align}\label{J1}
 \mathcal{J}_1(x,\varphi_1,\varphi_2):=\mathbb{E}_x\bigg[\int\limits_{0}^{\infty}e^{-r_1t}\pi( X^{\varphi_1,\varphi_2}_t)dt-\sum\limits_{k\geq 1}e^{-r_1{\tau_{1,k}}}g_1(\xi_k)\mathds{1}_{\{\tau_{1,k}<\infty\}}\bigg],
 \end{align}
 where $r_1>0$ is the subjective discount factor of the firm.\par
Furthermore, the government's total expected costs arising from the emissions of pollutants is 
  \begin{align}\label{J2}
  \mathcal{J}_2(x,\varphi_1,\varphi_2):=\mathbb{E}_x\bigg[\int\limits_{0}^{\infty}e^{-r_2t}C(\beta X^{\varphi_1,\varphi_2}_t)dt+\sum\limits_{k\geq 1}e^{-r_2{\tau_{2,k}}}g_2(\eta_k)\mathds{1}_{\{\tau_{2,k}<\infty\}}\bigg],
  \end{align}
  for some $r_2>0$ and $\beta>0$. The constant $\beta$ is the proportional factor between the rate of emissions and the output of production, while $r_2$ characterizes the time preferences of the government.
\begin{remark}
	We notice that the choice of a constant $\beta>0$ in \eqref{J2}, and of a constant $\alpha>0$ in \eqref{controlledX} is just to simplify exposition. Indeed, the results of this paper do hold even if we allow for suitable state dependent $\beta(\cdot)$ or $\alpha(\cdot)$.
\end{remark}

The firm and the government pick their policies within the following admissible class.
\begin{definition}\label{admissible}
	For any initial level of the production $x>0$, we say that the policies $\varphi_1:=(\tau_{1,1},\dots,\tau_{1,n},\dots;\xi_1,\dots,\xi_n,\dots)$ and $\varphi_2:=(\tau_{2,1},\dots,\tau_{2,n},\dots;\eta_1,\dots,\eta_n,\dots)$ are admissible, and we write $(\varphi_1,\varphi_2)\in\mathcal{S}(x)$, if the following hold true:
	\begin{itemize}
		\item [(i)] There exists a unique strong solution to \eqref{controlledX} with right-continuous sample paths such that $X^{x,\varphi_1,\varphi_2}_t\geq 0$  $\mathbb{P}$-a.s.\ for all $t\geq 0$.
		\item [(ii)] The functionals \eqref{J1} and \eqref{J2} are finite; that is,
		\begin{itemize}
			\item [(a)] $\displaystyle \mathbb{E}_x\bigg[\int_{0}^{\infty}e^{-r_1t}\pi( X^{\varphi_1,\varphi_2}_t)dt+ \int_{0}^{\infty}e^{-r_2t}C(\beta X^{\varphi_1,\varphi_2}_t)dt\bigg]<\infty $,
			\item [(b)] $\displaystyle \mathbb{E}_x\bigg[\sum\limits_{k\geq 1}e^{-r_1{\tau_{1,k}}}g_1(\xi_k)\mathds{1}_{\{\tau_{1,k}<\infty\}}+\sum\limits_{k\ge 1}e^{-r_2{\tau_{2,k}}}g_2(\eta_k)\mathds{1}_{\{\tau_{2,k}<\infty\}}\bigg]<\infty$.
		\end{itemize}
		\item [(iii)] If $\tau_{i,k}=\tau_{i,k+1}$ for some $i=1,2$ and $k\geq 1$, then $\tau_{i,k}=\tau_{i,k+1}=\infty$ $\mathbb{P}_x$-a.s.
		\item[(iv)] One has $\lim\limits_{k\rightarrow\infty}\tau_{i,k}=+\infty$ $\mathbb{P}_x$-a.s.\ for $i=1,2$.
	\end{itemize}
\end{definition}

Notice that requirements (iii) and (iv) prevent each agent to act twice at the same time, and to accumulate her interventions. In order to ensure that $\mathcal{S}(x)\neq\emptyset$, we now make the following standing assumption.
\begin{Assumptions}
\label{Ass1}
	The total expected profits and costs associated to non-intervention policies are such that
	$$\displaystyle \mathbb{E}_x\bigg[\int_{0}^{\infty}e^{-r_1t}\pi( X_t)dt+ \int_{0}^{\infty}e^{-r_2t}C(\beta X_t)dt\bigg]<\infty.$$
\end{Assumptions}
\noindent Indeed, under Assumption \ref{Ass1}, it follows that the policies associated with no interventions, i.e.\ $\tau_{i,k}=\infty$ $\mathbb{P}_x$-a.s.,\ for any $i=1,2$ and $k\geq 1$, belong to $\mathcal{S}(x)$.

\begin{remark}
Notice that in the benchmark cases in which the uncontrolled output of production is such that $dX_t=\mu X_tdt+\sigma X_tdW_t$, i.e.\ $X_t=x\exp\{(\mu-\frac{1}{2}\sigma^2)t+\sigma W_t\}$, $\mu\in\mathbb{R}$, $\sigma>0$, and $\pi(x)=x^a$, $a \in (0,1)$, and $C(x)=x^b$, $b>1$, one has that Assumption \ref{Ass1} is satisfied by taking 
$$r_1>\bigg[\mu a-\frac{\sigma^2a}{2}(1-a)\bigg]^+\quad\text{ and }\quad r_2> \bigg[\mu b+\frac{\sigma^2b}{2}(b-1)\bigg]^+.$$
\end{remark}

Given the policy adopted by the other agent, the firm aims at maximizing its profit, whereas the government at minimizing the social costs of pollution. Hence, for any $x>0$, the two agents aim at finding $(\varphi_1^\ast,\varphi_2^\ast)\in\mathcal{S}(x)$ such that
\begin{align}\label{NE}
\begin{cases}
\mathcal{J}_1(x,\varphi_1^\ast,\varphi_2^\ast)&\geq \,\,\mathcal{J}_1(x,\varphi_1,\varphi_2^\ast),\quad \forall \varphi_1 \text{ such\,\,that }(\varphi_1,\varphi^\ast_2)\in\mathcal{S}(x),\\
\mathcal{J}_2(x,\varphi^\ast_1,\varphi^\ast_2)&\leq \,\,\mathcal{J}_2(x,\varphi^\ast_1,\varphi_2),\quad \forall \varphi_2 \text{ such\,\,that }(\varphi_1^\ast,\varphi_2)\in\mathcal{S}(x).
\end{cases}\tag{$\mathcal{P}$}
\end{align}
\begin{definition}\label{problem}
	Let $x>0$. If $(\varphi_1^\ast,\varphi_2^\ast)\in\mathcal{S}(x)$ satisfying \eqref{NE} exist, we call them equilibrium policies, and we define the equilibrium values as
	\begin{align*}
		V_1(x):=\mathcal{J}_1(x,\varphi^\ast_1,\varphi^\ast_2)\quad\text{ and }\quad
		V_2(x):=\mathcal{J}_2(x,\varphi^\ast_1,\varphi^\ast_2).
	\end{align*}
\end{definition}


\section{Solving the Strategic Pollution Control Problem}
\label{sec:solution}
In this section, we first construct a pair of admissible candidate equilibrium policies. Then, under suitable requirements, we show that these policies indeed solve problem \eqref{NE}.


\subsection{Construction of a Candidate Solution}
\label{sec:representation}
We conjecture that a solution $(\varphi_1^\ast,\varphi_2^\ast)$ solving \eqref{NE} exists and is characterized by three intervals of the real line. These are the so-called \emph{joint} \emph{inaction region}, where both agents do not intervene on the production process, and the \emph{action regions} of the firm and of the government, where the two agents independently intervene on the output of production. More precisely, we conjecture the following.
\begin{itemize}
	\item [(i)] The firm increases its production instantaneously by exerting an impulse whenever the output of production is such that $X_t\leq b^{1}_1$, for some $b^{1}_1>0$ to be found, and shifts the process upwards to $b^{1}_2$, where $b^{1}_2>b^{1}_1$. We therefore define the candidate \emph{firm's} \emph{action region} as $\mathcal{A}_1:=(0,b^1_1]$.
	\item [(ii)] The government introduces regulatory constraints whenever the level of production, hence of emissions, is too large, i.e. $X_t\geq b^{2}_2$, for some $b^2_2$ to be determined, and induces a shift of the process downwards to some $b^{2}_1$, where $b^{2}_2>b^{2}_1>b^1_1$. Hence, the candidate \emph{government's} \emph{action region} is given by $\mathcal{A}_2:=[b^2_2,\infty)$.
\end{itemize}
In the rest of this paper, we will denote by $\mathcal{I}_i:=\mathbb{R}_+\setminus \mathcal{A}_i$ the \emph{inaction region} of agent $i$. Notice that the previous conjecture assumes that the constant barriers $b^i_j,\,i,j=1,2,$ of the government (resp.\ the firm) are decided \emph{ex-ante} and do not dynamically react to the policy followed by the firm (resp.\ government). Therefore, they trigger \emph{precommitted policies} of the two agents.

Following the previous conjecture, for any $x>0$ given and fixed we set
$$\tilde{\varphi}_1:=(\tilde{\tau}^{\tilde{\varphi}_1,\varphi_2}_{1,1},\dots,\tilde{\tau}^{\tilde{\varphi}_1,\varphi_2}_{1,n},\dots;\tilde{\xi}_1,\dots,\tilde{\xi}_n,\dots) \quad \text{and} \quad \tilde{\varphi}_2:=(\tilde{\tau}^{{\varphi}_1,\tilde{\varphi}_2}_{2,1},\dots,\tilde{\tau}^{{\varphi}_1,\tilde{\varphi}_2}_{2,n},\dots;\tilde{\eta}_1,\dots,\tilde{\eta}_n,\dots),$$ 
where we have introduced: 
\begin{itemize}
\item[(a)] the sequence of the firm's intervention times $\{\tilde{\tau}_{1,k}^{\tilde{\varphi}_1,\varphi_2}\}_{k \geq 1}$ such that $\tilde{\tau}_{1,k}^{\tilde{\varphi}_1,\varphi_2}:=\inf\{t>\tilde{\tau}^{\tilde{\varphi}_1,\varphi_2}_{1,k-1}: X_t^{x,\tilde{\varphi}_1,{\varphi}_2}\leq b^1_1\}$ for all $\varphi_2$ such that $(\tilde{\varphi}_1,\varphi_2)\in\mathcal{S}(x)$, and with $\tilde{\tau}^{\tilde{\varphi}_1,\varphi_2}_{1,0}:=0$ $\mathbb{P}$-a.s.; 
\item[(b)] the sequence of the government's intervention times $\{\tilde{\tau}_{2,k}^{{\varphi}_1,\tilde{\varphi}_2}\}_{k\geq 1}$ such that $\tilde{\tau}_{2,k}^{{\varphi}_1,\tilde{\varphi}_2}:=\inf\{t>\tilde{\tau}^{{\varphi}_1,\tilde{\varphi}_2}_{2,k-1}: X_t^{x,{\varphi}_1,\tilde{\varphi}_2}\geq b^2_2\}$ for all $\varphi_1$ such that $({\varphi}_1,\tilde{\varphi}_2)\in\mathcal{S}(x)$, and with $\tilde{\tau}^{{\varphi}_1,\tilde{\varphi}_2}_{2,0}:=0$ $\mathbb{P}$-a.s.; 
\item[(c)] the sequence of interventions of the firm $\tilde{\xi}_k:=\frac{1}{\alpha}(b^{1}_2- X^{x,\tilde{\varphi}_1,\varphi_2}_{\tilde{\tau}_{1,k}^{\tilde{\varphi}_1,\varphi_2}-})$ for all $k\geq 1$ and $\varphi_2$ such that $(\tilde{\varphi}_1,\varphi_2)\in\mathcal{S}(x)$; 
\item[(d)] the sequence of impulses applied by the government $\tilde{\eta}_k:=X^{x,\varphi_1,\tilde{\varphi}_2}_{\tilde{\tau}_{2,k}^{{\varphi}_1,\tilde{\varphi}_2}-}-b^{2}_1$ for all $k\geq1$ and $\varphi_1$ such that $({\varphi}_1,\tilde{\varphi_2})\in\mathcal{S}(x)$. 
\end{itemize}
By the definition of $\tilde{\tau}_{1,k}^{\tilde{\varphi}_1,\varphi_2}$ and $\tilde{\tau}_{2,k}^{{\varphi}_1,\tilde{\varphi}_2}$ one has that the sequence of impulses $\tilde{\xi}_k$ and $\tilde{\eta}_k$ are constant-sized (apart the initial impulses, that depend on the initial state $x$). In particular, $\tilde{\xi}_k:=(b^{1}_2- b^1_1)/\alpha$ and $\tilde{\eta}_k:=b^2_2-b^{2}_1$ for all $k\geq2$, and $\tilde{\xi}_1:=(b^{1}_2- x \wedge b^1_1)/\alpha$ and $\tilde{\eta}_1:= x \vee b^2_2-b^{2}_1$.

Notice that the policies $(\tilde{\varphi}_1,\tilde{\varphi}_2)$ just described exist because the constant trigger values $b^i_j$, $i,j=1,2$, of agent $i=1,2$ do not depend on the policy employed by agent $j\neq i$. That is, independently of the action of agent $j$, agent $i$ will always force the process $X$ to stay in her \emph{inaction region} $\mathcal{I}_i$. A rigorous formalization of $(\tilde{\varphi}_1,\tilde{\varphi}_2)$ can be obtained by the arguments  employed in Definition 2.2 of A{\"i}d et al.\ (2018). We now show that the policies $(\tilde{\varphi}_1,\tilde{\varphi}_2)$ previously defined are in fact admissible.
\begin{lemma}
	\label{lemma1}
	Recall Definition \ref{admissible}. Then for any $x>0$ the policies $(\tilde{\varphi}_1,\tilde{\varphi}_2)\in\mathcal{S}(x)$.
\end{lemma}

\begin{proof}
	Let $x > 0$ be given and fixed. Existence of a unique strong solution to \eqref{controlledX} with right-continuous paths can be obtained by arguing as in Lemma 2.3 of A{\"i}d et al.\ (2018). Also, $X_t^{x,\tilde{\varphi}_1,\tilde{\varphi}_2}\in[b^1_1,b^2_2] \subset [0,\infty)$ $\mathbb{P}$-a.s.\ for all $t>0$. Hence, Condition $(i)$ of Definition \ref{admissible} is satisfied. 
	
	The fact that $X_t^{x,\tilde{\varphi}_1,\tilde{\varphi}_2}\in[b^1_1,b^2_2]$ $\mathbb{P}$-a.s.\ for all $t>0$ and the continuity of $\pi$ and $C$ in particular imply that $(ii)-(a)$ of Definition \ref{admissible} is fulfilled. As for $(ii)-(b)$ note that $\tilde{\xi}_k\leq b^1_2/\alpha$ $\mathbb{P}_x$-a.s.\ for all $k\in \mathbb{N}$, and that $\tilde{\eta}_k\leq\max(b^2_2-b^2_1,x-b^2_1)$ $\mathbb{P}_x$-a.s.\ for all $k\in\mathbb{N}$. Hence there exists a positive constant $\Theta$ (possibly depending on $x$) such that $g_1(\tilde{\xi}_k) + g_2(\tilde{\eta}_k)\leq \Theta$ $\mathbb{P}_x$-a.s.\ for all $k\in\mathbb{N}$. In order to prove that $(ii)-(b)$ of Definition \ref{admissible} holds true, it thus suffices to show that for any $i=1,2$ one has
	\begin{align*}
	\mathbb{E}_x\bigg[\sum_{k\geq 1}e^{-r_i\tilde{\tau}^{\tilde{\varphi}_1,\tilde{\varphi}_2}_{i,k}}\bigg]<\infty.
	\end{align*}
To accomplish that one can adapt to our setting arguments from the proof of Proposition 4.7 in A{\"i}d et al.\ (2018). We provide these arguments here for the sake of completeness. Without loss of generality we consider the case $i=1$, since the treatment of the case $i=2$ is analogous. Defining $\tilde{\tau}:=\inf\{t>0: X_{t}^{b^1_2,\tilde{\varphi}_1,\tilde{\varphi}_2}\leq b^1_1\}$, and exploiting the time-homogeneity of the production process $X$ and the independence of the Brownian increments, we can write for any $k\geq 1$
	\begin{align*}
	\mathbb{E}_x\big[e^{-r_1\tilde{\tau}^{\tilde{\varphi}_1,\tilde{\varphi}_2}_{1,k}}\big]=\mathbb{E}_x\big[e^{-r_1\tilde{\tau}^{\tilde{\varphi}_1,\tilde{\varphi}_2}_{1,k-1}}\big]\mathbb{E}\big[e^{-r_1\tilde{\tau}}\big]. 
	\end{align*}
	By iterating the previous argument one finds $\mathbb{E}_x\big[e^{-r_1\tilde{\tau}^{\tilde{\varphi}_1,\tilde{\varphi}_2}_{1,k}}\big]=\mathbb{E}_x\big[e^{-r_1\tilde{\tau}^{\tilde{\varphi}_1,\tilde{\varphi}_2}_{1,1}}\big]\big(\mathbb{E}\big[e^{-r_1\tilde{\tau}}\big]\big)^{k-1}$. Then summing over $k$ on both sides of the previous equation and applying Fubini-Tonelli's theorem, we obtain
	\begin{align*}
	\mathbb{E}_x\bigg[\sum_{k\geq 1}e^{-r_1\tilde{\tau}^{\tilde{\varphi}_1,\tilde{\varphi}_2}_{1,k}}\bigg]=\mathbb{E}_x\big[e^{-r_1\tilde{\tau}^{\tilde{\varphi}_1,\tilde{\varphi}_2}_{1,1}}\big]\sum_{k\geq 0}\bigg(\mathbb{E}\big[e^{-r_1\tilde{\tau}}\big]\bigg)^k,
	\end{align*}
	and the series on the right-hand-side above converges as $\mathbb{E}[e^{-r_1\tilde{\tau}}]<1$.
	
	Finally, because $b^1_1<b^2_2$ by assumption, and $b^1_2,b^2_1\in(b^1_1,b^2_2)$, condition $(iii)$ and $(iv)$ of Definition \ref{admissible} are satisfied.
\end{proof}

The expected payoffs associated to the admissible policies $(\tilde{\varphi}_1,\tilde{\varphi}_2)$ are defined as $$v_1(x):=\mathcal{J}_1(x,\tilde{\varphi}_1,\tilde{\varphi}_2)\quad \text{and}\quad v_2(x):=\mathcal{J}_2(x,\tilde{\varphi}_1,\tilde{\varphi}_2),\quad x>0.$$
 Moreover, thanks to Assumption \ref{Ass1}, the performance criteria associated with no interventions are finite and given by
\begin{align}\label{G1G2}
 G_1(x):=\mathbb{E}_x\bigg[\int_{0}^{\infty}e^{-r_1s}\pi(X_s)ds\bigg]\quad \text{and}\quad G_2(x):=\mathbb{E}_x\bigg[\int_{0}^{\infty}e^{-r_2s}C(\beta X_s)ds\bigg].
\end{align}

For frequent future use, we define the infinitesimal generator $\LL_X$ of the uncontrolled diffusion $X^x$ by
$$\big(\LL_X u\big)(x):=\frac{1}{2}\sigma^2(x)u''(x) + \mu(x)u'(x), \quad \,\,x>0,$$
for any $u \in C^2((0,\infty))$. Then, for fixed $r>0$, under Assumption \ref{ass:D2} there always exist two linearly independent, strictly positive solutions to the ordinary differential equation $\LL_X u = r u$ satisfying a set of boundary conditions based on the boundary behaviour of $X^x$ (see, e.g., pp.\ 18--19 of Borodin and Salminen (2002)). These functions span the set of solutions of $\LL_X u = r u$, and are uniquely defined up to multiplication if one of them is required to be strictly increasing and the other one to be strictly decreasing. We denote the strictly increasing solution by $\psi_r$ and the strictly decreasing one by $\phi_r$. From now on we set $\psi_i:=\psi_{r_i}$ and $\phi_i:=\phi_{r_i}$ for $i=1,2$.

\begin{remark}
	The functions $G_1$ and $G_2$ are the expected cumulative present value of the flows $\pi(X^x_t)$ and $C(\beta X^x_t)$, respectively. It is well known that $G_i$, $i=1,2$, can be represented in terms of the fundamental solutions $\psi_{i}$ and $\phi_{i}$, $i=1,2$. We refer the reader to equation (3.3) in Alvarez (2004), among others.
\end{remark}

For any $i=1,2$ we introduce the strictly decreasing and positive function $F_i$ such that $F_i(x):=\phi_i(x)/\psi_i(x)$. Also, for given $b^i_j,\, i,j=1,2$, such that $0<b^{1}_1<b^{1}_2<b^{2}_2$ and $b^{1}_1<b^{2}_1<b^{2}_2$, we set
\begin{align}\label{AB}
A_i(x):=\frac{\psi_i(x)}{\psi_i(b^1_1)}\Bigg[\frac{F_i(b^2_2)-F_i(x)}{F_i(b^2_2)-F_i(b^1_1)}\Bigg],\quad
B_i(x):=\frac{\psi_i(x)}{\psi_i(b^2_2)}\Bigg[\frac{F_i(x)-F_i(b^1_1)}{F_i(b^2_2)-F_i(b^1_1)}\Bigg] \quad i=1,2.
\end{align}

We define $w_i$ as the restriction of $v_i$ on $\mathcal{I}_1\cap\mathcal{I}_2$, i.e.\ $w_i:=v_i|_{\mathcal{I}_1\cap\mathcal{I}_2}$. The next result provides a representation of $v_i(x)=\mathcal{J}_i(x,\tilde{\varphi}_1,\tilde{\varphi}_2)$, $i=1,2$. 
\begin{proposition}
\label{prop1}
Recall \eqref{AB}, let $x>0$, and $b^i_j,\, i,j=1,2$, such that $0<b^{1}_1<b^{1}_2<b^{2}_2$ and $b^{1}_1<b^{2}_1<b^{2}_2$.
Then, the performance criteria $v_1(x)$ and $v_2(x)$ associated to the policies $(\tilde{\varphi}_1,\tilde{\varphi}_2)\in\mathcal{S}(x)$ can be represented as
\begin{align}\label{w1}
v_1(x)=
\begin{cases}
w_1(b^{1}_2)-K_1-\frac{\kappa_1}{\alpha}(b^{1}_2-x),\vspace{1.5mm}\quad &x\leq b^{1}_1,\\
\big[w_1(b^{1}_2)-K_1-\frac{\kappa_1}{\alpha}(b^1_2-b^1_1)-G_1(b^1_1)\big]A_1(x)\vspace{1.5mm}\\
+\big[w_1(b^2_1)-G_1(b^2_2)\big]B_1(x)+G_1(x),\vspace{1.5mm}\quad &x\in(b^1_1,b^2_2),\\
w_1(b^{2}_1), &x\geq b^{2}_2,
\end{cases}
\end{align}
and 
\begin{align}\label{w2}
v_2(x)=
\begin{cases}
w_2(b^{1}_2),\vspace{1.5mm}\quad &x\leq b^{1}_1\\
\big[w_2(b^2_1)+K_2+\kappa_2(b^2_2-b^2_1)-G_2(b^2_2)\big]B_2(x)\vspace{1.5mm}\\
+\big[w_2(b^1_2)-G_2(b^1_1)\big]A_2(x)+G_2(x),\vspace{1.5mm} &x\in(b^{1}_1,b^{2}_2),\\
w_2(b^{2}_1)+K_2+\kappa_2(x-b^{2}_1), &x\geq b^{2}_2.
\end{cases}
\end{align}
Moreover, under the requirement
\begin{align}
\label{cond1}
\big(1-A_i(b^1_2)\big)\big(1-B_i(b^2_1)\big)-B_i(b^1_2)A_i(b^2_1)\neq 0,\quad i=1,2,
\end{align}
one has 
\begin{align}
\label{cond2}
\begin{split}
w_1(b^1_2)=&\bigg[1-A_1(b^1_2)-\frac{B_1(b^1_2)A_1(b^2_1)}{1-B_1(b^2_1)}\bigg]^{-1}\bigg[\frac{G_1(b^2_1)B_1(b^1_2)}{1-B_1(b^2_1)}+G_1(b^1_2)\\
&-\bigg(K_1+\kappa_1(b^1_2-b^1_1)+G_1(b^1_1)\bigg)\bigg(\frac{A_1(b^2_1)B_1(b^1_2)}{1-B_1(b^2_1)}+A_1(b^1_2)\bigg)\\
&-G_1(b^2_2)\bigg(\frac{B_1(b^2_1)B_1(b^1_2)}{1-B_1(b^2_1)}+B_1(b^1_2)\bigg)\bigg],
\end{split}
\end{align}
\begin{align}\label{cond3}
\begin{split}
w_1(b^2_1)=&\big[1-B_1(b^2_1)\big]^{-1}\big[\big(w_1(b^1_2)-K_1-\kappa_1(b^1_2-b^1_1)-G_1(b^1_1)\big)A_1(b^2_1)\\
&-G_1(b^2_2)B(b^2_1)+G_1(b^2_1)\big],
\end{split}
\end{align}
and
\begin{align}
\label{cond4}
\begin{split}
w_2(b^1_2)=&\bigg[\frac{\big(1-A_2(b^1_2)\big)\big(1-B_2(b^2_1)\big)}{B_2(b^1_2)}-A_2(b^2_1)\bigg]^{-1}\times\\
&\bigg[\frac{G_2(b^1_2)\big(1-B_2(b^2_1)\big)}{B_2(b^1_2)}+G_2(b^2_1)+K_2+\kappa_2(b^2_2-b^2_1)-G_2(b^2_2)\\
&-G_2(b^1_1)\bigg(A_2(b^1_2)\frac{1-B_2(b^2_1)}{B_2(b^1_2)}+A_2(b^2_1)\bigg)\bigg],
\end{split}
\end{align}
\begin{align}
\label{cond5}
\begin{split}
w_2(b^2_1)=&\big[1-B_2(b^2_1)\big]^{-1}\big[\big(K_2+\kappa_2(b^2_2-b^2_1)-G_2(b^2_2)\big)B_2(b^2_1)\\
&+\big(w_2(b^1_2)-G_2(b^1_1)\big)A_2(b^2_1)+G_2(b^2_1)\big].
\end{split}
\end{align}
\end{proposition}

\begin{proof}
We consider only the case $i=1$ since the arguments are symmetric for $i=2$. Let $x>0$ be given and fixed, and define $\tau_1:=\inf\{t\geq 0:X^x_t\leq b^{1}_1\}$ and $\tau_2:=\inf\{t\geq 0:X^x_t\geq b^{2}_2\}$. According to the policies $(\tilde{\varphi}_1,\tilde{\varphi}_2)$, the stopping time $\tau_1 \wedge \tau_2$ is the first time at which either the firm or the government intervenes. Then, noticing that $X$ is uncontrolled up to time $\tau_1 \wedge \tau_2$, the payoff of the firm associated to $(\tilde{\varphi}_1,\tilde{\varphi}_2)$ satisfies the functional relation 
\begin{align}\label{P1}
\begin{split}
 v_1(x)=\mathbb{E}_x\bigg[&\int\limits_{0}^{\tau_{1}\wedge\tau_{2}}e^{-r_1t}\pi(X_t)dt +e^{-r_1\tau_1}\mathds{1}_{\{\tau_1<\tau_2\}}\Big(v_1(b^{1}_2)-K_1-\frac{\kappa_1}{\alpha}(b^{1}_2-X^{\tilde{\varphi}_1,\tilde{\varphi}_2}_{\tau_1})\Big)\\
 &+e^{-r_1\tau_2}\mathds{1}_{\{\tau_1 > \tau_2\}}v_1(b^{2}_1)\bigg].
\end{split}
\end{align}
Recall that $w_i$ denotes the restriction of $v_i$ on $\mathcal{I}_1\cap\mathcal{I}_2$. Then, taking $x\in(b^1_1,b^2_2)=\mathcal{I}_1\cap\mathcal{I}_2$ in \eqref{P1}, noticing that $b^1_2$ and $b^2_1$ belong to $\mathcal{I}_1\cap\mathcal{I}_2$ and recalling \eqref{G1G2}, by the strong Markov property we can write
\begin{align*}
w_1(x)=
&\big(w_1(b^{1}_2)-K_1-\frac{\kappa_1}{\alpha}(b^1_2-b^1_1)-G_1(b^1_1)\big)\mathbb{E}_x\big[e^{-r_1\tau_1}\mathds{1}_{\{\tau_1<\tau_2\}}\big]\\[1.5mm]
&+\big(w_1(b^2_1)-G_1(b^2_2)\big)\mathbb{E}_x\big[e^{-r_1\tau_2}\mathds{1}_{\{\tau_1 > \tau_2\}}\big]+G_1(x).
\end{align*}
By using now the formulas for the Laplace transforms of hitting times of a linear diffusion (see, e.g., Dayanik and Karatzas (2003), eq.\ (4.3)), we find (cf. \eqref{AB})
\begin{align*}
\mathbb{E}_x\big[e^{-r_1\tau_1}\mathds{1}_{\{\tau_1<\tau_2\}}\big]=A_1(x),\quad
\mathbb{E}_x\big[e^{-r_1\tau_2}\mathds{1}_{\{\tau_1 > \tau_2\}}\big]=B_1(x),
\end{align*}
so that 
$$w_1(x)=\big(w_1(b^{1}_2)-K_1-\frac{\kappa_1}{\alpha}(b^1_2-b^1_1)-G_1(b^1_1)\big)A_1(x)+\big(w_1(b^2_1)-G_1(b^2_2)\big)B_1(x)+G_1(x),$$
for all $x \in (b^1_1,b^2_2)$.

Taking $x\leq b^1_1$ in \eqref{P1} we obtain $\tau_1=0$ and then $v_1(x)=w_1(b^{1}_2)-K_1-\frac{\kappa_1}{\alpha}(b^{1}_2-x)$, while picking $x\geq b^2_2$ we have $\tau_2=0$ and thus $v_1(x)=w_1(b^2_1)$. Therefore we can write
\begin{align}\label{repv1}
v_1(x)=
\begin{cases}
w_1(b^{1}_2)-K_1-\frac{\kappa_1}{\alpha}(b^{1}_2-x),\vspace{1.5mm}\quad &x\leq b^{1}_1,\\
\big[w_1(b^{1}_2)-K_1-\frac{\kappa_1}{\alpha}(b^1_2-b^1_1)-G_1(b^1_1)\big]A_1(x)\vspace{1.5mm}\\
+\big[w_1(b^2_1)-G_1(b^2_2)\big]B_1(x)+G_1(x),\vspace{1.5mm}\quad &x\in(b^1_1,b^2_2),\\
w_1(b^{2}_1), &x\geq b^{2}_2.
\end{cases}
\end{align}

Recall \eqref{cond1} and that $b^1_2,b^2_1\in(b^1_1,b^2_2)$ by construction. Then, taking first $x=b^1_2$ and then $x=b^2_1$ in \eqref{repv1}, we obtain a linear system of two equations for the two unknowns $w_1(b^1_2)$ and $w_1(b^2_1)$. Once solved, this system yields
\begin{align*}
w_1(b^1_2)=&\bigg[1-A_1(b^1_2)-\frac{B_1(b^1_2)A_1(b^2_1)}{1-B_1(b^2_1)}\bigg]^{-1}\bigg[\frac{G_1(b^2_1)B_1(b^1_2)}{1-B_1(b^2_1)}+G_1(b^1_2)\\
&-\bigg(K_1+\kappa_1(b^1_2-b^1_1)+G_1(b^1_1)\bigg)\bigg(\frac{A_1(b^2_1)B_1(b^1_2)}{1-B_1(b^2_1)}+A_1(b^1_2)\bigg)\\
&-G_1(b^2_2)\bigg(\frac{B_1(b^2_1)B_1(b^1_2)}{1-B_1(b^2_1)}+B_1(b^1_2)\bigg)\bigg],
\end{align*}
and 
\begin{align*}
w_1(b^2_1)=&\big[1-B_1(b^2_1)\big]^{-1}\big[\big(w_1(b^1_2)-K_1-\kappa_1(b^1_2-b^1_1)-G_1(b^1_1)\big)A_1(b^2_1)\\
&-G_1(b^2_2)B(b^2_1)+G_1(b^2_1)\big].
\end{align*}
Notice that the denominators in the definition of $w_1(b^1_2)$ are nonzero. Indeed, $B_1(b^2_1) \neq 1$ since $\tau_2 > 0$ $\mathbb{P}$-a.s.\ for $x=b^2_1 < b^2_2$, and $(1-A_1(b^1_2))(1-B_1(b^2_1))-B_1(b^1_2)A_1(b^2_1)\neq 0$ by \eqref{cond1}.

The proof is then completed.
\end{proof}

It is easy to see from \eqref{w1} and \eqref{w2} that $v_i,\, i=1,2,$ is by construction a continuous function on $(0,\infty)$. In order to obtain the four boundaries $b^i_j$, $i,j=1,2$, we first assume that each agent picks her own action boundary $b^i_i,\, i=1,2,$ such that $v_i$ is also continuously differentiable there. This gives
\begin{align}\label{eq1}
v_1^\prime(b^1_{1}\,+)&=\frac{\kappa_1}{\alpha},\\\label{eq2}
v_2^\prime(b^2_{2}\,-)&=\kappa_2,
\end{align}
where we have set $v_i^\prime(\cdot \,\pm\,):=\lim_{\varepsilon \downarrow 0}v_i^\prime(\,\cdot \,\pm \varepsilon)$.

The two equations \eqref{eq1} and \eqref{eq2} may be interpreted as the so-called \emph{smooth-fit} equations, well known optimality conditions in the literature on singular/impulse control and optimal stopping (see, e.g., Fleming and Soner (2005) and Peskir and Shiryaev (2006)).
Furthermore, we assume that at each intervention the firm and the government shift the process $X$ to the points that give rise to the maximal net profits and minimal total costs, respectively. This means that  $b_2^1,b_1^2\in(b^1_1,b^2_2)$ are selected such that
\begin{align*}
b^1_2={\arg\sup}_{y\geq b^1_1}\big\{v_1(y)-\frac{\kappa_1}{\alpha}(y-x)-K_1\big\},\quad x\leq b^1_1,
\end{align*}
and 
\begin{align*}
b^2_1={\arg\inf}_{y\leq b^2_2}\big\{v_2(y)+\kappa_2(x-y)+K_2\big\},\quad x\geq b^2_2.
\end{align*}
Consequently, 
	\begin{align}
	\label{eq3}
	v_1^\prime(b^1_2)&=\frac{\kappa_1}{\alpha},\\\label{eq4}
	v_2^\prime(b^2_1)&=\kappa_2.
	\end{align}
	
The \emph{four} equations \eqref{eq1}-\eqref{eq4} can be used in order to obtain the \emph{four} unknowns $b^1_1, b^1_2, b^2_1, b^2_2$, whenever a solution to such a highly nonlinear system exists.

\subsection{The Verification Theorem}
\label{sec:verification}

Here we prove a verification theorem providing a set of sufficient conditions under which the solution to \eqref{eq1}-\eqref{eq4} (if it exists) characterizes an equilibrium; that is, $(\tilde{\varphi}_1,\tilde{\varphi}_2)=(\varphi^\ast_1,\varphi^\ast_2)$, and $v_1\equiv V_1$, $v_2\equiv V_2$ (cf. Definition \ref{problem}). For its proof the following assumption is needed.

\begin{Assumptions}
\label{ass:theta}
\begin{itemize}\hspace{10cm}
\item[(i)]  There exists $\hat{x}_1>0$ such that the function $\theta_1:\mathbb{R}_+\mapsto\mathbb{R}$ with $\theta_1(x):= \pi(x) + \frac{\kappa_1}{\alpha}(\mu(x) - r_1 x)$ attains a local maximum at $\hat{x}_1$ and is increasing on $(0,\hat{x}_1)$;
\item[(ii)] There exists $\hat{x}_2>0$ such that the function $\theta_2:\mathbb{R}_+\mapsto\mathbb{R}$ with $\theta_2(x):= C(\beta x) + \kappa_2(\mu(x) - r_2 x)$ attains a local minimum at $\hat{x}_2$ and is increasing on $(\hat{x}_2,\infty)$.
\end{itemize}
\end{Assumptions}

\begin{remark}
It is worth noticing that Assumption \ref{ass:theta} is verified by the benchmark cases $\mu(x)=\mu x$, $\mu\in\mathbb{R}$, $\pi(x)=x^a$, $a\in(0,1)$, and $C(x)=x^b$, $b>1$, with $\hat{x}_1=\big[\frac{\kappa_1}{a\alpha}(r_1-\mu)\big]^{\frac{1}{a-1}}$, $\hat{x}_2=\big[\frac{\kappa_2}{b\beta^b}(r_2-\mu)\big]^{\frac{1}{b-1}}$ (whenever $r_1 \wedge r_2 > \mu$).
\end{remark}

\begin{theorem} [Verification Theorem]
\label{Verification}
Let Assumption \ref{ass:theta} hold. Let $b^{i}_j$, $i,j=1,2$, be a solution of \eqref{eq1}-\eqref{eq4} such that $0<b^{1}_1<b^{1}_2<b^{2}_2$, $b^{1}_1<b^{2}_1<b^{2}_2$ and satisfying \eqref{cond1}, recall $v_1$, $v_2$ as in \eqref{w1} and \eqref{w2}, and suppose that
	\begin{align}\label{Cond6}
	v_1^\prime(x)&\geq \frac{\kappa_1}{\alpha}, &\text{for all}\,\, x\in(b^1_1,b^1_2],\\
	\label{Cond3}
	v_1^\prime(x)&<\frac{\kappa_1}{\alpha}, &\text{for all}\,\, x\in(b^1_2,b^2_2],\\
	\label{Cond4}
	v_2^\prime(x)&<\kappa_2,&\text{for all}\,\, x\in(b^1_1,b^2_1),\\\label{Cond5}
	v_2^\prime(x)&\geq \kappa_2, &\text{for all}\,\, x\in[b^2_1,b^2_2),
	\end{align}
	and
	\begin{align}
	\label{hx1}
	b^1_1&\leq\hat{x}_1,\\ \label{hx1bis} \pi(b^1_1)+\frac{c_1}{\alpha}\mu(b^1_1) - r_1v_1(b^1_1)&\leq 0,\\\label{hx2}
	b^2_2&\geq\hat{x}_2,\\ \label{hx2bis} C(\beta b^2_2)+\kappa_2\mu(b^2_2) - r_2v_2(b_2^2)&\geq 0.
	\end{align}
	Then, for $x>0$, the policies $(\tilde{\varphi}_1,\tilde{\varphi}_2)\in\mathcal{S}(x)$ such that
	 \begin{align}
	 \,
	 \begin{cases}
	 \displaystyle \tilde{\tau}^{\tilde{\varphi}_1,\tilde{\varphi}_2}_{i,k}=\inf\{t>\tilde{\tau}_{i,k-1}: X_t^{\tilde{\varphi}_1,\tilde{\varphi}_2}\notin \mathcal{I}_i\},\quad &k\geq 1,\,\,\mathbb{P}_x-a.s.,\\
	 \displaystyle \tau_{i,0}^{\tilde{\varphi}_1,\tilde{\varphi}_2}=0, &\text{$\mathbb{P}_x$-a.s.,}
	 \end{cases}  
	 \end{align}	
	for $i=1,2$, and
	\begin{equation}
	\label{impulses}
	\tilde{\xi}_k=\frac{1}{\alpha}\Big(b^{1}_2-X_{\tilde{\tau}^{\tilde{\varphi}_1,\tilde{\varphi}_2}_{1,k}-}^{\tilde{\varphi}_1,\tilde{\varphi}_2}\Big), \qquad \tilde{\eta}_k =X_{\tilde{\tau}^{\tilde{\varphi}_1,\tilde{\varphi}_2}_{2,k}-}^{\tilde{\varphi}_1,\tilde{\varphi}_2}-b^{2}_1, \qquad k\geq 1, \qquad \text{$\mathbb{P}_x$-a.s.,}
	\end{equation}
	form an equilibrium, and $v_1$ and $v_2$ are the corresponding equilibrium values; that is,
	\begin{align*}
	v_1 = V_1,\quad v_2 = V_2 \quad \text{on}\,\,\,(0,\infty).
	\end{align*}
\end{theorem}
\begin{proof}
	The proof is organized in two steps. \vspace{0.25 cm}
	
		\emph{Step 1.}	Here we discuss the regularity properties of the function $v_i$, $i=1,2$, constructed in Proposition \ref{prop1}. Note that by \eqref{w1} and \eqref{w2} one can directly check that $v_i\in C((0,\infty))$ for $i=1,2$. Moreover, by \eqref{eq1} and \eqref{eq2} one has $v_1\in C^1((0,b^2_2))$, $v_2\in C^1((b^1_1,\infty))$ and it can be checked by direct calculations that $v_1^{\prime\prime}\in L^\infty_{\text{loc}}((0,b^2_2))$ and $v_2^{\prime\prime}\in L^\infty_{\text{loc}}((b^1_1,\infty))$.		
		Also, for any $x\in(b^1_1,b^2_2)$ we have from \eqref{w1} and \eqref{w2} that $\big(\mathcal{L}_Xv_1-r_1v_1\big)(x)+\pi(x)=0$ and $\big(\mathcal{L}_Xv_2-r_2v_2\big)(x)+C(\beta x)=0$.
		
		Because $\theta_1$ is increasing on $(0,\hat{x}_1)$ (cf.\ Assumption \ref{ass:theta}), and $b^1_1\leq\hat{x}_1$ by assumption, we obtain from \eqref{w1} that for any $x < b^1_1$ one has
		\begin{align}\label{L1}
		\begin{split}
		&\big(\mathcal{L}_Xv_1-r_1v_1\big)(x)+\pi(x)=\theta_1(x)-r_1\big(v_1(b^1_2)-K_1-\frac{\kappa_1}{\alpha} b^1_2\big)\\
		&\leq \theta_1(b^1_1)-r_1\big(v_1(b^1_2)-K_1-\frac{\kappa_1}{\alpha} b^1_2\big)=\pi(b^1_1)+\frac{\kappa_1}{\alpha}\mu(b^1_1)-r_1 v_1(b^1_1)\leq 0,
		\end{split}
		\end{align}
		where we have used that $v_1(b^1_2)=v_1(b^1_1)+K_1+\frac{\kappa_1}{\alpha}(b^1_2-b^1_1)$, \eqref{hx1} and \eqref{hx1bis}.
		
		 Similarly, one can check that $\big(\mathcal{L}_Xv_2-r_2v_2\big)(x)+C(\beta x)\geq0$  for all $x>b_2^{2}$ due to \eqref{hx2}, \eqref{hx2bis}, and the fact that $\theta_2$ is increasing on $(\hat{x}_2,\infty)$ (cf.\ Assumption \ref{ass:theta}).
		\vspace{0.25cm}
		
		\emph{Step 2.} Given $x>0$ we now prove that $(\tilde{\varphi}_1,\tilde{\varphi}_2)\in\mathcal{S}(x)$ are equilibrium policies;  that is,
		\begin{align*}
		v_1(x)\geq\mathcal{J}_1(x,\varphi_1,\tilde{\varphi}_2),\quad \forall \varphi_1 \text{ s.t. }(\varphi_1,\tilde{\varphi}_2)\in\mathcal{S}(x),\\ v_2(x)\leq \mathcal{J}_2(x,\tilde{\varphi}_1,\varphi_2),\quad \forall \varphi_2 \text{ s.t. }(\tilde{\varphi}_1,\varphi_2)\in\mathcal{S}(x),
		\end{align*}
		with equalities when we pick $\varphi_1=\tilde{\varphi}_1$ and $\varphi_2=\tilde{\varphi}_2$.
Without loss of generality we consider $i=1$, since the arguments for $i=2$ are analogous.

Let $\varphi_1=(\tau_{1,1},\dots,\tau_{1,n},\dots;\xi_{1},\dots,\xi_{n},\dots)$ be such that $(\varphi_1,\tilde{\varphi}_2)\in\mathcal{S}(x)$, and for $N>0$ set $\tau_{R,N}:=\tau_{R}\wedge N,$ where $\tau_{R}:=\inf\{s>0:X^{x,\varphi_1,\tilde{\varphi}_2}_s\notin (-R,R)\}$, with the usual convention $\inf\emptyset=\infty$. Since $X^{x,\varphi_1,\tilde{\varphi}_2}_t\leq b^2_2$ $\mathbb{P}$-a.s.\ for all $t>0$, by the regularity of $v_1$ discussed in \emph{Step 1} we can apply the generalized It\^o's formula for semimartingales (see, e.g., {\O}ksendal and Sulem (2006), Theorems 2.1 and 6.2), so to obtain
\begin{align}
\label{1}
 v_1(x)=\mathbb{E}_x\Bigg[&-\int\limits_{0}^{\tau_{R,N}}e^{-r_1t}(\mathcal{L}_Xv_1-r_1v_1)(X^{\varphi_1,\tilde{\varphi}_2}_t)dt+e^{-r_1\tau_{R,N}}v_1(X^{\varphi_1,\tilde{\varphi}_2}_{\tau_{R,N}})\nonumber\\
 &-\sum\limits_{k:\,\tau_{1,k}<\tau_{R,N}}e^{-r_1\tau_{1,k}}\big(v_1(X^{\varphi_1,\tilde{\varphi}_2}_{\tau_{1,k}})-v_1(X^{\varphi_1,\tilde{\varphi}_2}_{\tau_{1,k}-})\big) \nonumber \\
&-\sum\limits_{k:\,\tilde{\tau}^{{\varphi}_1,\tilde{\varphi}_2}_{2,k}<\tau_{R,N}}e^{-r_1\tilde{\tau}^{{\varphi}_1,\tilde{\varphi}_2}_{2,k}}\big(v_1(X^{\varphi_1,\tilde{\varphi}_2}_{\tilde{\tau}^{{\varphi}_1,\tilde{\varphi}_2}_{2,k}})-v_1(X^{\varphi_1,\tilde{\varphi}_2}_{\tilde{\tau}^{{\varphi}_1,\tilde{\varphi}_2}_{2,k}-})\big)\Bigg].
\end{align}
By using again that $X^{x,\varphi_1,\tilde{\varphi}_2}_t \leq b^2_2$ for all $t>0$ $\mathbb{P}$-a.s., and since $(\mathcal{L}_Xv_1-r_1v_1)(x)\leq-\pi( x)$ for a.e. $x<b^2_2$ due to \eqref{L1}, we obtain from \eqref{1} that
\begin{align}
\label{2}
v_1(x)\geq\mathbb{E}_x\Bigg[&\int\limits_{0}^{\tau_{R,N}}e^{-r_1t}\pi(X^{\varphi_1,\tilde{\varphi}_2}_t)dt \,-\sum\limits_{k:\,\tau_{1,k}<\tau_{R,N}}e^{-r_1\tau_{1,k}}\big(v_1(X^{\varphi_1,\tilde{\varphi}_2}_{\tau_{1,k}})-v_1(X^{\varphi_1,\tilde{\varphi}_2}_{\tau_{1,k}-})\big) \nonumber \\
&-\sum\limits_{k:\,\tilde{\tau}^{{\varphi}_1,\tilde{\varphi}_2}_{2,k}<\tau_{R,N}}e^{-r_1\tilde{\tau}^{{\varphi}_1,\tilde{\varphi}_2}_{2,k}}\big(v_1(X^{\varphi_1,\tilde{\varphi}_2}_{\tilde{\tau}^{{\varphi}_1,\tilde{\varphi}_2}_{2,k}})-v_1(X^{\varphi_1,\tilde{\varphi}_2}_{\tilde{\tau}^{{\varphi}_1,\tilde{\varphi}_2}_{2,k}-})\big) +e^{-r_1\tau_{R,N}}v_1(X^{\varphi_1,\tilde{\varphi}_2}_{\tau_{R,N}})\Bigg].
\end{align}

In order to take care of the two sums in the expectation above, we define the nonlocal operator
\begin{align*}
\big(\mathcal{M}_1 v_1\big)(x)&:=\sup_{\xi \geq 0}\big\{v_1(x+\alpha\xi)-g_1(\xi)\big\},
\end{align*}
and we notice that $\tilde{\xi}_k$ of \eqref{impulses} is such that $\tilde{\xi}_k=\argsup_{\xi\geq 0}\big\{v_1(x+\alpha\xi)-g_1(\xi)\big\}$, for all $k \in \mathbb{N}$, due to \eqref{Cond6} and \eqref{Cond3}. Hence
\begin{align}
\label{intervention}
\big(\mathcal{M}_1v_1\big)(x)=
\begin{cases}
v_1(b^1_2)-K_1-\frac{\kappa_1}{\alpha}(b^1_2-x),\quad&\text{if } x\leq b^1_2,\\[1.5mm]
v_1(x)-K_1,\quad&\text{if } x> b^1_2.
\end{cases}
\end{align}

One can easily see from \eqref{w1} and \eqref{intervention} that $v_1(x)\geq \big(\mathcal{M}_1v_1\big)(x)$ for all $x\in(0,b^1_1]\cup(b^1_2,\infty)$, with equality for $x \leq b^1_1$. Then, noticing that $x \mapsto v_1(x)- \big(\mathcal{M}_1v_1\big)(x)$ is increasing for any $x \in (b^1_1, b^1_2]$ by \eqref{Cond6} and \eqref{intervention}, we conclude that $v_1(x)\geq \big(\mathcal{M}_1v_1\big)(x)$ for all $x>0$. Therefore
\begin{align}
\label{3}
v_1(X^{x,\varphi_1,\tilde{\varphi}_2}_{\tau_{1,k}-})\geq\big(\mathcal{M}_1v_1\big)(X^{x,\varphi_1,\tilde{\varphi}_2}_{\tau_{1,k}-})\geq v_1\big(X^{x,\varphi_1,\tilde{\varphi}_2}_{\tau_{1,k}}\big)-g_1(\xi_k),
\end{align}
for any $\mathcal{F}_{\tau_{1,k}}$-measurable $\xi_k\geq 0$. Moreover, because $X^{x,\varphi_1,\tilde{\varphi}_2}_{\tilde{\tau}^{{\varphi}_1,\tilde{\varphi}_2}_{2,k}-}\geq b^2_2$ $\mathbb{P}$-a.s.\ and $X^{x,\varphi_1,\tilde{\varphi}_2}_{\tilde{\tau}^{{\varphi}_1,\tilde{\varphi}_2}_{2,k}}=b^2_1$ $\mathbb{P}$-a.s., we find by \eqref{w1} that
\begin{align}\label{4}
v_1(X^{x,\varphi_1,\tilde{\varphi}_2}_{\tilde{\tau}^{{\varphi}_1,\tilde{\varphi}_2}_{2,k}-})=v_1(X^{x,\varphi_1,\tilde{\varphi}_2}_{\tilde{\tau}^{{\varphi}_1,\tilde{\varphi}_2}_{2,k}}),
\end{align}
upon noticing that $v_1(b^2_1) = w_1(b^2_1)$ since $b^2_1 \in (b^1_1, b^2_2)$. It thus follows from \eqref{3} and \eqref{4} that
\begin{equation}
\label{5}
v_1(x)\geq\mathbb{E}_x\Bigg[\int\limits_{0}^{\tau_{R,N}}e^{-r_1t}\pi( X^{x,\varphi_1,\tilde{\varphi}_2}_t)dt-\sum\limits_{k:\tau_{1,k}<\tau_{R,N}}e^{-r_1\tau_{1,k}}g_1(\xi_k) + e^{-r_1\tau_{R,N}}v_1(X^{x,\varphi_1,\tilde{\varphi}_2}_{\tau_{R,N}})\Bigg].
\end{equation}

But now, $v$ is continuous and $X_t^{x,\varphi_1,\tilde{\varphi}_2}\in[0,b^2_2]$ $\mathbb{P}$-a.s.\ by admissibility of $(\varphi_1,\tilde{\varphi}_2)$. Hence, we can write 
$$e^{-r_1\tau_{R,N}}v_1(X^{x,\varphi_1,\tilde{\varphi}_2}_{\tau_{R,N}}) \geq - e^{-r_1\tau_{R,N}}|v_1(X^{x,\varphi_1,\tilde{\varphi}_2}_{\tau_{R,N}})| \geq - e^{-r_1\tau_{R,N}}\max_{x \in [0,b^2_2]}|v_1(x)|,$$
and from \eqref{5} we have
\begin{equation}
\label{6}
v_1(x)\geq\mathbb{E}_x\Bigg[\int\limits_{0}^{\tau_{R,N}}e^{-r_1t}\pi( X^{x,\varphi_1,\tilde{\varphi}_2}_t)dt-\sum\limits_{k:\tau_{1,k}<\tau_{R,N}}e^{-r_1\tau_{1,k}}g_1(\xi_k) - e^{-r_1\tau_{R,N}}\max_{x \in [0,b^2_2]}|v_1(x)|\Bigg].
\end{equation} 
By using the dominated convergence theorem for the last term in \eqref{6} and the monotone convergence theorem for the integral and the series in \eqref{6}, we let first $R\rightarrow\infty$ and then $N\rightarrow\infty$, and we find
\begin{align*}
v_1(x)\geq \mathcal{J}_1(x,\varphi_1,\tilde{\varphi}_2).
\end{align*}
Finally, by construction we also have $v_1(x)=\mathcal{J}_1(x,\tilde{\varphi}_1,\tilde{\varphi}_2).$

Because arguments analogous to the ones employed for $v_1$ yield $v_2(x)\leq\mathcal{J}_2(x,\tilde{\varphi}_1,\varphi_2)$ for all $\varphi_2$ such that $(\tilde{\varphi}_1,\varphi_2)\in\mathcal{S}(x)$, and $v_2(x)=\mathcal{J}_2(x,\tilde{\varphi}_1,\tilde{\varphi}_2)$, we conclude that $(\tilde{\varphi}_1,\tilde{\varphi}_2)$ are equilibrium policies and $(v_1,v_2)$ are the corresponding equilibrium values.
\end{proof}

\begin{remark}
As a byproduct of Theorem \ref{Verification} we have that, if \eqref{Cond6}-\eqref{hx2bis} are fulfilled, then $v_1$ and $v_2$ satisfy in the a.e.\ sense the system of QVIs 
\begin{align}\label{7}
\begin{split}
\max\{\big(\mathcal{L}v_1-r_1v_1\big)(x)+\pi( x),\,\mathcal{M}_1v_1(x)-v_1(x)\}=0,\quad &\text{for a.e. } x< b^{2}_2,\\
\min\{\big(\mathcal{L}v_2-r_2v_2\big)(x)+C(\beta x),\,\mathcal{M}_2v_2(x)-v_2(x)\}=0,\quad &\text{for a.e. } x> b^{1}_1,\\
v_1(x)\geq \mathcal{M}_1v_1(x),\quad&\forall x>0,\\
v_2(x)\leq \mathcal{M}_2v_2(x),\quad&\forall x>0,\\
v_1(x)=v_1\big(b^{2}_1\big),\quad &\forall x\geq b^{2}_2,\\
v_2(x)=v_2\big(b^{1}_2\big),\quad &\forall x\leq b^{1}_1.
\end{split}	
\end{align}
A system analogous to \eqref{7} has been introduced in the context of nonzero-sum stochastic differential games with impulse controls in A{\"i}d et al.\ (2018).
\end{remark}


\section{A Numerical Example and Comparative Statics}
\label{sec:numeric}
Verification Theorem \ref{Verification} involves the highly nonlinear system of \emph{four} algebraic equations \eqref{eq1}--\eqref{eq4} for the \emph{four} boundaries. We have solved this system numerically in a specific setting by using MATLAB. In particular, for the numerical example we have assumed that the uncontrolled output of production evolves as a geometric Brownian motion, i.e.\ $\mu(x) = \mu x$ and $\sigma(x) = \sigma x $ for some $\mu \in \mathbb{R}$ and $\sigma>0$. Moreover, we have taken an operating profit function of Cobb-Douglas type $\pi(x) = x^a$, $a \in (0,1)$, and a social disutility function of the form $C(x) = x^b$, $b>1$.

Among the possible parameters' values satisfying Assumption \ref{Ass1}, we pick for example those provided in Table \ref{table},
\begin{table}[h]
	\centering
	\begin{tabular}   {| l | l | l | l | l | l | l | l | l | l | l | l |}
		\hline
		$\mu$ & $\sigma$ & $r_1$ & $r_2$ & $\alpha$ & $\beta$ & $K_1$ & $\kappa_1$ & $K_2$ & $\kappa_2$ & $a$ & $b$ \\ \hline
		0.02 & 0.20 & 0.10 & 0.10 & 1 & 1 & 0.5 & 0.8 & 0.6 & 0.3& 0.5 & 2 \\
		\hline
	\end{tabular}
	\caption{Parameters' values for the numerical example.}
	\label{table}
\end{table}
and we notice that for such a choice the performance criteria associated with no interventions (cf.\ \eqref{G1G2}) are given by
\begin{align}
G_1(x)=\frac{1}{r_1 - \frac{\mu}{2} + \frac{\sigma^2}{8}}\sqrt{x} = \frac{1000}{95}\sqrt{x},\quad \text{and}\quad
G_2(x)=\frac{1}{r_2 - 2\mu - \sigma^2}x^2 = 50 x^2.
\end{align}

Also, by an application of the Newton method in MATLAB, we find that the numerical solution to \eqref{eq1}-\eqref{eq4} is given by
\begin{align*}
b_1^{1}&=0.1558984470,\quad b_2^{1}=0.3825673799,\\ b_1^{2}&=0.2359455020,\quad b_2^{2}=0.5746537199,
\end{align*}
where we have evaluated $w_i(b^1_2)$ and $w_i(b^2_1),i=1,2,$ by \eqref{cond2}-\eqref{cond5}. One also finds (cf.\ \eqref{hx1}-\eqref{hx2bis})
\begin{align*}
\hat{x}_1&=\big[2\kappa_1(r_1-\mu)\big]^{-2}=61.03515625>b^1_1,\quad &\pi(b^1_1)+\frac{c_1}{\alpha}\mu(b^1_1)-r_1v_1(b^1_1)=-0.0727643376\leq0,\\
\hat{x}_2&=\frac{\kappa_2(r_2-\mu)}{2}=0.012<b^2_2,\quad &C(\beta b^2_2)+\kappa_2\mu(b^2_2)-r_2v_2(b_2^2)=0.1390988361\geq0.
\end{align*}
The plots of the equilibrium values and of their derivatives in the joint inaction region $(b^1_1,b^2_2)$ are provided in Figures \ref{figure1}, \ref{figure2}, and \ref{figure3} and \ref{figure4}, respectively. In Figures \ref{figureSA} and \ref{figureSAG} one observes the drawings of the value functions that the firm and the government would have in a non-strategic setting (i.e.\ if the two agents optimize their own performance criterion in absence of the other agent).
\begin{figure}[htbp] 
	\subfigure[\label{figure1}Equilibrium value $V_1$ in $(b^1_1,b^2_2)$.]{\includegraphics[width=0.5\textwidth]{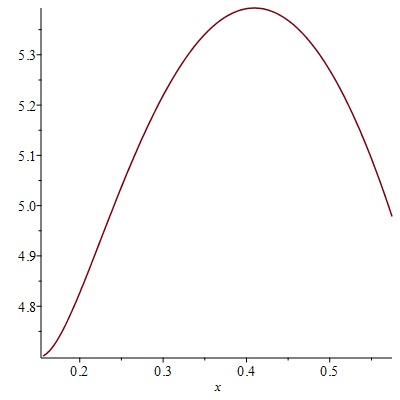}}
	\subfigure[\label{figure2}Equilibrium value $V_2$ in $(b^1_1,b^2_2)$.]{\includegraphics[width=0.5\textwidth]{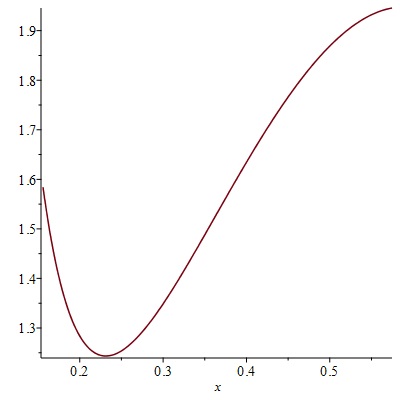}}
	\subfigure[\label{figureSA}Value function of the firm in the \emph{inaction} region for a non-strategic model.]{\includegraphics[width=0.5\textwidth]{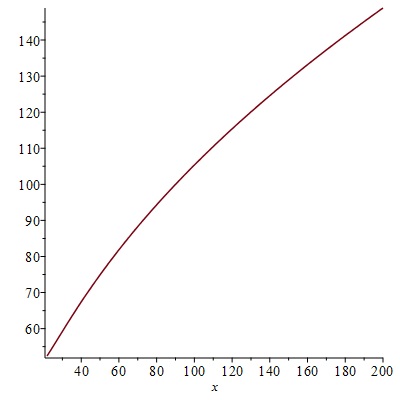}}
	\subfigure[\label{figureSAG}Value function of the government in the \emph{inaction} region for a non-strategic model.]{\includegraphics[width=0.5\textwidth]{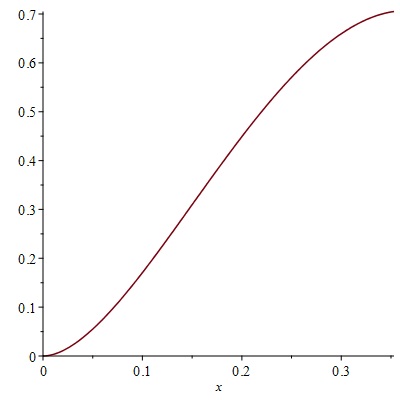}}
	\caption{Value functions in the strategic and non-strategic setting.}
\end{figure}

Comparing Figures \ref{figure1} and \ref{figure2} with Figures \ref{figureSA} and \ref{figureSAG}, one can notice that the value functions that the two agents would have in a non-strategic setting are monotone with respect to the state variable. On the contrary, the equilibrium values $V_1$ and $V_2$ are not monotone functions, and this is clearly a consequence of the strategic interaction between the two agents. From Figures \ref{figure3} and \ref{figure4} one can also check that conditions \eqref{Cond6}-\eqref{Cond5} are satisfied.
 \begin{figure}[htbp] 
 	\subfigure[\label{figure3}Derivative of $V_1$.]{\includegraphics[width=0.5\textwidth]{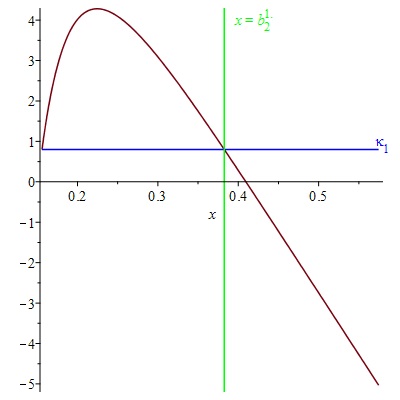}}
 	\subfigure[\label{figure4}Derivative of $V_2$.]{\includegraphics[width=0.5\textwidth]{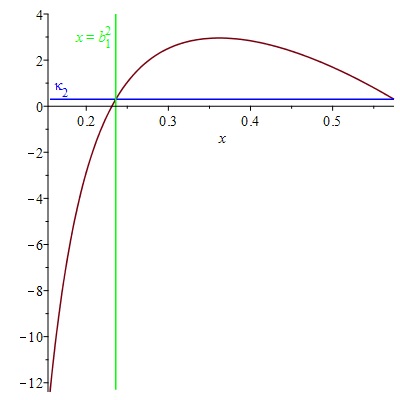}}
 	\caption{Derivatives of the equilibrium values.}
 \end{figure}

We now discuss the dependency of our equilibrium policies with respect to the model parameters. The following plots are obtained with MATLAB through an application of the Newton method initialized at the parameters' values specified in Table \ref{table} above.

Figure \ref{figure5} displays the behavior of the optimal action boundaries $b^1_1$ and $b^2_2$ when the volatility $\sigma$ varies in the range $[0.19,0.22]$. Furthermore, Figure \ref{figure6} shows how the optimal size of interventions, $b^1_2-b^1_1$ and $b^2_2-b^2_1$, changes with $\sigma$.
\begin{figure}[htbp] 
	\subfigure[\label{figure5}The optimal action boundaries $b^1_1$ (black), $b^1_2$ (blue), $b^2_1$ (red), $b^2_2$ (green).]{\includegraphics[width=0.5\textwidth]{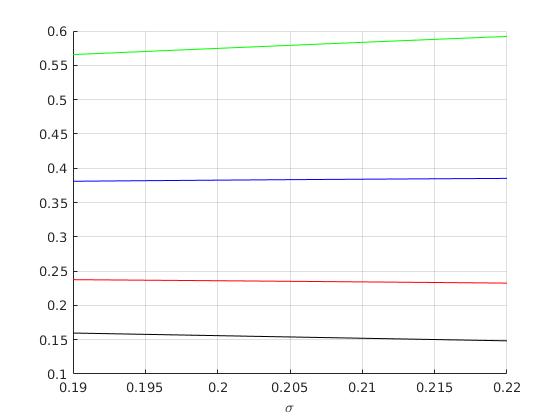}}
	\subfigure[\label{figure6}Optimal size of interventions: firm (blue) and government (black).]{\includegraphics[width=0.5\textwidth]{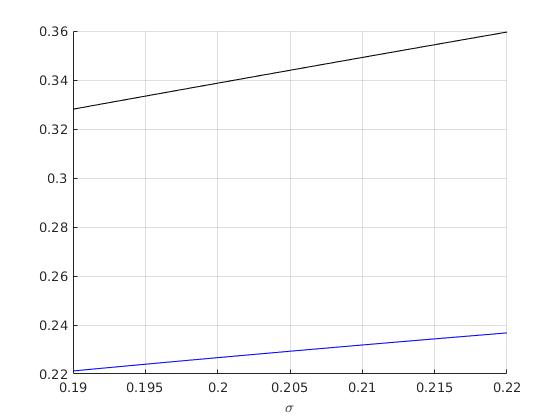}}
	\caption{Dependency of the equilibrium on the volatility $\sigma$.}
		\subfigure[\label{figure7}The optimal action boundaries $b^1_1$ (black), $b^1_2$ (blue), $b^2_1$ (red), $b^2_2$ (green).]{\includegraphics[width=0.5\textwidth]{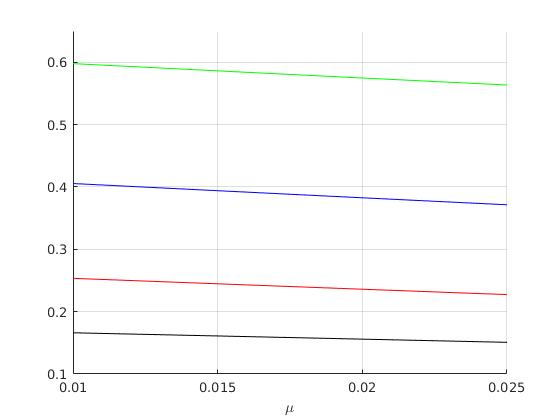}}
	\subfigure[\label{figure8}Optimal size of interventions: firm (blue) and government (black).]{\includegraphics[width=0.5\textwidth]{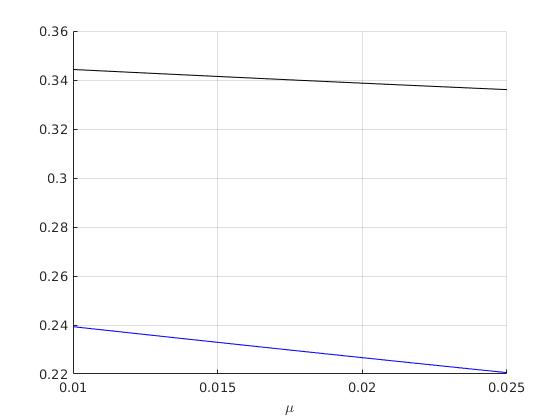}}
	\caption{Dependency of the equilibrium on the drift $\mu$.}
\end{figure}
One can observe that the optimal action threshold of the government increases with $\sigma$, whereas the firm's action threshold decreases. This behavior is well-known in the real options literature (see the seminal paper by McDonald and Siegel (1986)): when uncertainty increases, the agent is more reluctant to act and her inaction region becomes larger. Furthermore, Figure \ref{figure6} reveals that the strength of interventions of the firm and of the government increases with increasing volatility. The higher are the fluctuations of the production/pollution process, the more the agents are afraid of a quicker need of a new costly intervention. Hence both the agents increase the size of their impulses in order to postpone their next action.

We now take $\sigma=0.2$, and we let $\mu$ vary in the interval $[0.01,0.025]$.
\begin{figure}[htbp] 
	\subfigure[\label{figure9}The optimal action boundaries $b^1_1$ (black), $b^1_2$ (blue), $b^2_1$ (red), $b^2_2$ (green).]{\includegraphics[width=0.5\textwidth]{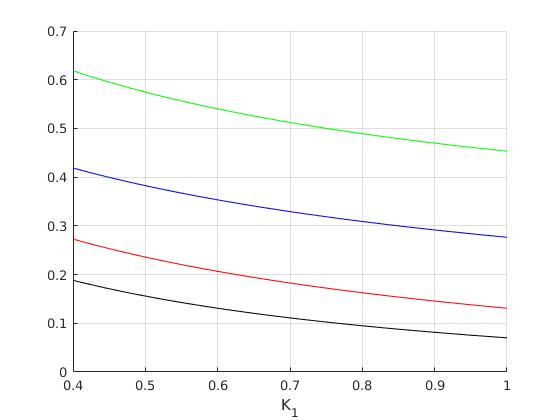}}
	\subfigure[\label{figure10}Optimal size of interventions: firm (blue) and government (black).]{\includegraphics[width=0.5\textwidth]{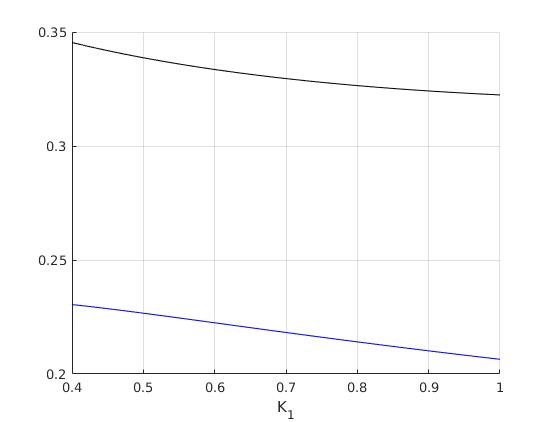}}
	\caption{Dependency of the equilibrium on the firm's fixed cost $K_1$.}
	\label{figure13}
	\subfigure[\label{figure11}The optimal action boundaries $b^1_1$ (black), $b^1_2$ (blue), $b^2_1$ (red), $b^2_2$ (green).]{\includegraphics[width=0.5\textwidth]{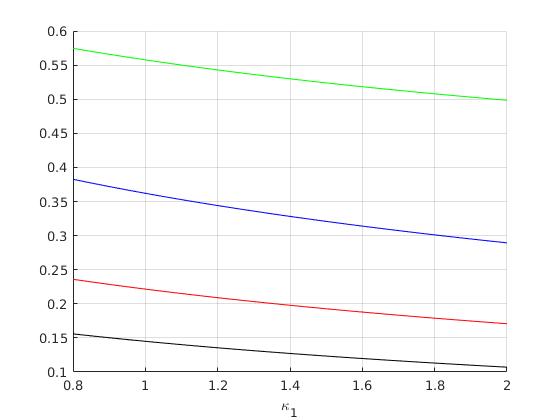}}
	\subfigure[\label{figure12}Optimal size of interventions: firm (blue) and government (black).]{\includegraphics[width=0.5\textwidth]{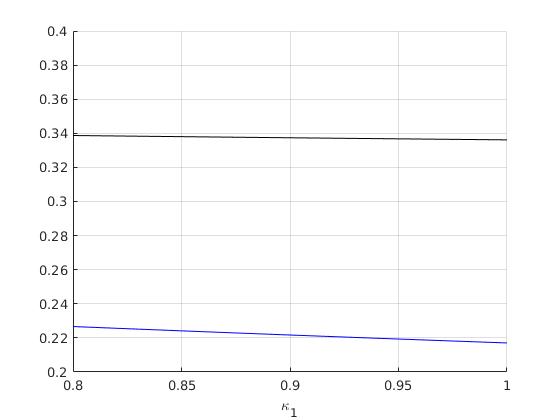}}
	\caption{Dependency of the equilibrium on the firm's variable cost $\kappa_1$.}
	\label{figure14}
\end{figure}
Figure \ref{figure7} leads us to the following conclusion: as the  drift $\mu$ increases, the firm's action region becomes smaller. That is, a higher trend of the output of production decreases the firm's willingness to intervene. We can also observe from Figure \ref{figure7} that the government's threshold decreases with $\mu$: since the output of production, and therefore the rate of emissions, increases faster, the government tries to dam the increasing social cost by introducing more severe regulatory constraints. Figure \ref{figure8} shows that the higher the trend of the output of production is, the lower is the size of interventions $b^1_2-b^1_1$, i.e.\ the lower the willingness of the firm to pay for additional capacity. Also, one can observe that the government's size of interventions decrease with increasing $\mu$. We believe that this effect is due to the strategic interactions between the two agents, and it might be justified as follows. The higher $\mu$ is, the smaller is the length of the joint inaction region (see Figure \ref{figure7}). Hence, the government reduces the size of interventions when $\mu$ increases so to likely reduce the firm's incentive to intervene.

Finally, we analyze the dependency of the action thresholds and of the equilibrium impulses' size with respect to the cost components $K_1$ and $\kappa_1$ (see Figures \ref{figure13} and \ref{figure14}). Similar behaviors are also observed with respect to $K_2$ and $\kappa_2$.
Higher fixed costs lead to decreasing action boundaries, see Figure \ref{figure9}, and therefore to a larger inaction region of the firm. As a consequence, the government exploits the firm's reluctance to invest when fixed costs are larger and confines the production process below a lower level. A particular comment is deserved by Figure \ref{figure10} where we observe that the sizes of interventions of both agents are decreasing with respect to $K_1$. This behavior might be explained once more as an effect of the strategic interaction between the two agents. When $K_1$ increases, the firm reduces the size of its interventions in order to likely avoid a possible further action by the government, and, in turn, a further costly capacity expansion. As a result of the reduction of the joint inaction region (see Figure \ref{figure9}), the government also diminishes its size of interventions so to try to prevent the firm from undertaking a further capacity expansion. A similar rationale might also explain the behavior of the equilibrium thresholds and equilibrium impulses' sizes with respect to the variable costs $\kappa_1$.


\section{Conclusions}
\label{conclusions}

In this paper a government and a firm, representative of the productive sector of a country, are the two players of a stochastic nonzero-sum game of impulse control. The firm faces both proportional and fixed costs to expand its stochastically fluctuating production with the aim of maximizing its expected profits. The government introduces regulatory constraints with the aim of reducing the level of emissions of pollutants and of minimizing the related total expected costs. Assuming that the emissions' level is proportional to the output of production, by issuing environmental policies the government effectively forces the firm to decrease its production.

We have conjectured that an equilibrium in this strategic problem is characterized by four constant trigger values, to be endogenously determined. We have then provided a set of sufficient conditions under which these candidate equilibrium policies do indeed form an equilibrium. Finally, we have studied numerically the case in which the (uncontrolled) output of production evolves as a geometric Brownian motion, and the firm's operating profit and the government's running cost function are of power type. Within such a setting, a study of the dependency of the equilibrium policies and values on the model parameters have yielded interesting new behaviors that we have explained as a result of the strategic interaction between the firm and the government.

There are many directions in which it would be interesting to extend the present study. As an example, one might consider a two-dimensional formulation of our game in which the state variables are given by the production capacity of the firm and the level of pollution. The firm faces a costly capacity expansion and maximizes its net expected profits. The output of production, however, increases the emissions, which in turn contribute to the accumulation of a pollution stock. The government aims at reducing the level of the pollution stock by issuing costly environmental policies. This would lead to a daunting two-dimensional stochastic game with impulse controls for which a sophisticated theoretical and numerical analysis might be needed.

\vspace{1cm}

\indent \textbf{Acknowledgments.} We thank two anonymous referees for their pertinent remarks and suggestions. We also wish to thank Giorgia Callegaro, Herbert Dawid, Frank Riedel and Jan-Henrik Steg for useful comments. Financial support by the German Research Foundation (DFG) through the Collaborative Research Centre 1283 ``Taming uncertainty and profiting from randomness and low regularity in analysis, stochastics and their applications'' is gratefully acknowledged by the authors.



\begin{thebibliography}{199}


\bibitem{aidetal} \textsc{A{\"i}d, R., Basei, M., Callegaro, G., Campi, L., Vargiolu, T.}\ (2018). Nonzero-sum Stochastic Differential Games with Impulse Controls: a Verification Theorem with Applications. Preprint. \textbf{ArXiv}: 1605.00039.

\bibitem{Alvarez}\textsc{Alvarez, L.H.R.}\ (2004). A Class of Solvable Impulse Control Problems. \textsl{Appl.\ Math.\ Optim.}~\textbf{49} 265--295.

\bibitem{AlvarezLempa}\textsc{Alvarez, L.H.R., Lempa, J.}\ (2008). On the Optimal Stochastic Impulse Control of Linear Diffusions. \textsl{SIAM J.\ Control Optim.}~\textbf{47(2)} 703--732.

\bibitem{Asea}\textsc{Asea, P.K., Turnovsky, S.J.}\ (1998). Capital Income Taxation and Risk-taking in a Small Open Economy. \textsl{J.\ Public Econ.\ Theory}~\textbf{68(1)} 55--90.

\bibitem{BL}\textsc{Bensoussan, A., Lions, J.-L.}\ (1984). \textsl{Impulse Control and Quasi-variational Inequalities}. Gauthier-Villars.

\bibitem{Bensoussan010} \textsc{Bensoussan, A., Moussawi-Haidar, L., \c{C}akanyildirim, M.}\ (2010). Inventory Control with an Order-time Constraint: Optimality, Uniqueness and Significance. \textsl{Ann.\ Oper.\ Res.}\ \textbf{181(1)} 603--640.

\bibitem{Bertola}\textsc{Bertola, G.}\ (1998). Irreversible Investment. \textsl{Res.\ Econ.}~\textbf{52(1)} 3--37. 

\bibitem{BS}\textsc{Borodin, A.N., Salminen, P.}\ (2002). \textsl{Handbook of Brownian Motion-Facts and Formulae} 2nd edition. Birkh\"auser.

\bibitem{Cadenillas} \textsc{Cadenillas, A., Zapatero, F.}\ (1999). Optimal Central Bank Intervention in the Foreign Exchange Market. \textsl{J.\ Econom.\ Theory} \textbf{87} 218--242.

\bibitem{Cadenillasetal} \textsc{Cadenillas, A., Choulli, T., Taksar, M., Zhang, L.}\ (2006). Classical and Impulse Stochastic Control for the Optimization of the Dividend and Risk Policies of an Insurance Firm. \textsl{Math.\ Finance} \textbf{16} 181--202.


\bibitem{DayKar}\textsc{Dayanik, S., Karatzas, I.}\ (2003). On the Optimal Stopping Problem for One-Dimensional Diffusions. \textsl{Stochastic Process.~Appl.}~\textbf{107(2)} 173--212.


\bibitem{DeAFe} \textsc{De Angelis, T., Ferrari, G.}\ (2016). Stochastic Nonzero-sum Games: a New Connection between Singular Control and Optimal Stopping. \textbf{ArXiv}: 1601.05709. Forthcoming on \textsl{Adv.~Appl.\ Probab.}

\bibitem{eaton} \textsc{Eaton, J.}\ (1981). Fiscal Policy, Inflation and the Accumulation of Risky Capital. \textsl{Rev.\ Econ.\ Studies} \textbf{48(153)} 435--445.

\bibitem{egami} \textsc{Egami, M.}\ (2008). A Direct Solution Method for Stochastic Impulse Control Problems of One-dimensional Diffusions. \textsl{SIAM J.\ Control Optim.} \textbf{47(3)} 1191--1218.

\bibitem{Epaul}\textsc{Epaulard, A., Pommeret, A.}\ (2003). Recursive Utility, Endogenous Growth, and the Welfare Cost of Volatility. \textsl{Review\ Econ.\ Dyn.}~\textbf{6(3)} 672--684.

\bibitem{fleming}\textsc{Fleming, W.H., Soner, H.M.}\ (2005). Controlled Markov Processes and Viscosity Solutions. 2nd Edition. Springer.

\bibitem{Goulder} \textsc{Goulder, L.H., Mathai, K.}\ (2000). Optimal CO$_2$ Abatement in the Presence of Induced Technological Change. \textsl{J.\ Environ.\ Econ.\ Manag.}\ \textbf{39} 1--38.


\bibitem{Harrison}\textsc{Harrison, J.M., Selke, T.M., Taylor, A.J.}\ (1983). Impulse Control of Brownian Motion \textsl{Math.\ Oper.\ Res.}\ \textbf{8(3)} 454--466.


\bibitem{KS} \textsc{Karatzas, I., Shreve, S.E.}~(1998). \textsl{Brownian Motion and Stochastic Calculus} 2nd Edition. Springer.

\bibitem{korn} \textsc{Korn, R.}\ (1999). Some Applications of Impulse Control in Mathematical Finance. \textsl{Math.\ Meth.\ Oper.\ Res.} \textbf{50} 493--528.

\bibitem{Long} \textsc{Long, N.V.}\ (1992). Pollution control: A differential game approach. \textsl{Ann.\ Oper.\ Res.}\ \textbf{37(1)} 283--296.

\bibitem{mcdonald} \textsc{McDonald, R.L., Siegel, D.R.}\ (1986). The Value of Waiting to Invest. \textsl{Q.\ J.\ Econ.}\ \textbf{101(4)} 707--728.


\bibitem{Mitchell} \textsc{Mitchell, D., Feng, H., Muthuraman, K.}\ (2014). Impulse Control of Interest Rates. \textsl{Oper.\ Res.} \textbf{62(3)} 602--615.

\bibitem{nordhaus} \textsc{Nordhaus, W.D.}\ (1994). \textsl{Managing the Global Commons.} Cambridge, Mass.: MIT Press.

\bibitem{OkSu} \textsc{{\O}ksendal, B., Sulem, A.}\ (2006). \textsl{Applied Stochastic Control of Jump Diffusions} 2nd Edition. Springer.

\bibitem{Perera} \textsc{Perera, S., Buckley, W., Long, H.}\ (2016). Market-reaction-adjusted Optimal Central Bank Intervention Policy in a Forex Market with Jumps. \textsl{Ann.\ Oper.\ Res.}\ doi:10.1007/s10479-016-2297-y.

\bibitem{Perman} \textsc{Perman, R., Ma, Y., McGilvray, J., Common, M.}\ (2003). \textsl{Natural Resource and Environmental Economics} 3rd edition. Pearson, Addison Wesley.

\bibitem{peskir}\textsc{Peskir, G., Shiryaev, A.}\ (2006). On the American Option Problem. \textsl{Math.\ Finance}~\textbf{26(2)} 331--349.

\bibitem{pigou} \textsc{Pigou, A.C.}\ (1938). \textsl{The Economics of Welfare} 4th Edition. Macmillan.

\bibitem{pindyck2} \textsc{Pindyck, R.S.}\ (2000). Irreversibilities and the Timing of Environmental Policy. \textsl{Res.\ Energy Econ.} \textbf{22} 233--259.

\bibitem{pindyck} \textsc{Pindyck, R.S.}\ (2002). Optimal Timing Problems in Environmental Economics. \textsl{J.\ Econ.\ Dyn.\ Control} \textbf{26(9-10)} 1677--1697.

\bibitem{pommeretprieur} \textsc{Pommeret, A., Prieur, F.}\ (2013). Double Irreversibility and Environmental Policy Timing. \textsl{J.\ Public Econ.\ Theory} \textbf{15(2)} 273--291.

\bibitem{Schwoon} \textsc{Schwoon, M., Tol, R.S.J.}\ (2006). Optimal CO$_2$-abatement with Socio-economic Inertia and Induced Technological Change. \textsl{Energy J.}\ \textbf{27(4)} 25--59.

\bibitem{VDP}\textsc{van der Ploeg, F., de Zeeuw, A.}\ (1991). A Differential Game of International Pollution Control. \textsl{Syst.\ Control Lett.}\ \textbf{17(6)} 409--414.

\bibitem{walde}\textsc{Wälde, K.}\ (2011). Production Technologies in Stochastic Continuous Time Models. \textsl{J.\ Econ.\ Dyn.\ Control}~\textbf{35(4)} 616--622.

\end{thebibliography}
\end{document}